%% file: periodic_main.tex
\documentclass[a4paper]{amsart} 

\input{preamble}

\usepackage{soul}
\usepackage{multirow,booktabs}
\usepackage{scrextend}
\usepackage[outline]{contour}
\usepackage[ruled, vlined]{algorithm2e}
\usepackage{soul}
\usepackage{tabularx}

\newcommand{\SCI}{\operatorname{SCI}}

\renewcommand{\i}{\mathrm{i}}

\newcommand{\diag}{\operatorname{diag}}

\newcommand{\cU}{\mathcal{U}}
\newcommand{\nf}{\nicefrac}
\newcommand{\cof}{\operatorname{cof}}

\title{Universal algorithms for computing spectra of periodic operators}
\author{Jonathan Ben-Artzi}
\email{Ben-ArtziJ@cardiff.ac.uk}
\author{Marco Marletta}
\email{MarlettaM@cardiff.ac.uk}
\author{Frank R\"{o}sler}
\email{RoslerF@cardiff.ac.uk}
\thanks{JBA  acknowledges support from an Engineering and Physical Sciences Research Council Fellowship EP/N020154/1. MM acknowledges support from an Engineering and Physical Sciences Research Council Grant EP/T000902/1. FR acknowledges support from the European Union's Horizon 2020 Research and Innovation Programme under the Marie Sklodowska-Curie grant agreement No 885904.}
\address{School of Mathematics, Cardiff University, Senghennydd Road, Cardiff CF24 4AG, Wales, UK}
\date\today

\begin{document}
\maketitle

\begin{abstract}
Schr\"odinger operators with periodic (possibly complex-valued) potentials and discrete periodic  operators (possibly with complex-valued entries) are considered, and in both cases the computational spectral problem is investigated: namely, under what conditions can a `one-size-fits-all' algorithm for computing their spectra be devised? It is shown that for periodic banded matrices this can be done, as well as for Schr\"odinger operators with periodic potentials that are sufficiently smooth. In both cases implementable algorithms are provided, along with examples. For certain Schr\"odinger operators whose potentials may diverge at a single point (but are otherwise well-behaved)  it is shown that there does not exist such an algorithm, though it is shown that the computation is possible if one allows for two successive limits.
\end{abstract}
%
%

\section{Introduction and Main Results}
We study the  computational spectral problem for periodic  discrete operators, acting in $\ell^2(\Z)$, as well as Schr\"odinger operators with  periodic potentials acting in $L^2(\R^d)$.  We show that it is possible to compute their respective spectra as limits of finite-dimensional approximations. However, in the Schr\"odinger case this becomes \emph{im}possible if the potential is allowed to be discontinuous at a \emph{single} point (but otherwise it is  smooth). More precisely, we prove:

\begin{theorem}[Periodic banded matrices]\label{thm:main1}
Let $B(\ell^2(\Z))$ denote the set of bounded operators on $\ell^2(\Z)$ and let $\{e_i\}_{i\in\Z}$ be a basis for $\ell^2(\Z)$. Let $\Omega^{\mathrm{per}}\subset B(\ell^2(\Z))$ be the class of banded, periodic operators with respect to the basis $\{e_i\}_{i\in\Z}$ and $\Omega^{\mathrm{per}}_{N,b}$ the subset of matrices with period $N$ and bandwidth $b$. Then 
\begin{enumi}
	\item
	there exists an algorithm that can compute the spectrum $\sigma(A)$ of any $A\in\Omega^{\mathrm{per}}$ as the limit of a sequence computable approximations;
	\item
	there exists an algorithm that can compute the spectrum $\sigma(A)$ of any $A\in\Omega^{\mathrm{per}}_{N,b}$ with guaranteed error bounds.
\end{enumi}
\end{theorem}
\begin{theorem}[Schr\"odinger: good case]\label{th:main2}
For $d\in\N$ define the class of potentials
\begin{align*}
	\Omega^{\mathrm{Sch}} &:= \{ V:\R^d\to\C \,|\, V \text{ is 1-periodic and }V|_{(0,1)^d}\in W^{1,p}((0,1)^d) \text{ for some }p>d \},
\end{align*}
and given $M>0$ and $p>d$  define the class
\begin{align*}
	\Om^{\mathrm{Sch}}_{p,M} &:= \{V\in\Om^{\mathrm{Sch}}\,|\, \|V\|_{W^{1,p}((0,1)^d)}\leq M\}.
\end{align*}
Then
	\begin{enumi}
	\item
	there exists an algorithm that can compute the spectrum $\sigma(H)$ of any Schr\"odinger operator  $H=-\Delta+V$ with $V\in\Omega^{\mathrm{Sch}}$ as the limit of a sequence computable approximations;
	\item
	moreover, for $V\in\Om^{\mathrm{Sch}}_{p,M}$ this algorithm yields  spectral inclusion with guaranteed error bounds.
	\end{enumi}
\end{theorem}

\begin{theorem}[Schr\"odinger: bad case]\label{thm:main3}
Let $x_0\in[0,1]$ and let
\begin{equation*}
	\Om^{\mathrm{Sch}}_{x_0} := \left\{ V:\R\to\R \,\middle|\, V \text{ is 1-periodic, } V(x_0)=0 \text{ and }V|_{[0,1]}\in L^2\big([0,1]\big)\cap C^\infty\big([0,1]\setminus\{x_0\}\big) \right\}.
\end{equation*}
Then there \underline{does not} exist an algorithm that can compute the spectrum $\sigma(H)$ of any Schr\"odinger operator  $H=-\Delta+V$ with $V\in\Omega^{\mathrm{Sch}}_{x_0}$ as the limit of a sequence computable approximations. However, there \underline{does} exist an algorithm that can compute $\sigma(H)$ by taking \underline{two} successive  limits.
\end{theorem}
Below, in Section \ref{sec:sci}, we give precise definitions of what an `algorithm' is,what information is available to it, how it computes, and hence what we mean when we say `computable'.
Informally, Theorems \ref{thm:main1} and \ref{th:main2} imply that these computations can be performed numerically. In fact, we provide  actual algorithms which can access the matrix entries (in   Theorem \ref{thm:main1}) and pointwise evaluations of the potential (in Theorem \ref{th:main2}). This should be contrasted with Theorem  \ref{thm:main3} which implies that it is \emph{impossible} to devise an algorithm that can compute the spectrum of \emph{any} Schr\"odinger operator with a potential belonging to $\Omega^{\mathrm{Sch}}_{x_0}$. We prove this by contradiction: assuming the existence of such an algorithm, we explicitly construct a  potential $V\in \Omega^{\mathrm{Sch}}_{x_0}$ for which this algorithm would fail in its attempt to compute the spectrum.

These statements are nontrivial. The existence of a `one-size-fits-all' algorithm as in Theorems  \ref{thm:main1} and \ref{th:main2} is not obvious: while there are techniques for computing the spectrum of a given operator in each of the above cases -- these are indeed well-studied problems -- here we prove the existence of a single algorithm (which can be coded) that can handle any input from the given class, without any additional \emph{a priori} information. Moreover, in the cases of $\Omega^{\mathrm{per}}_{N,b}$ and $\Om^{\mathrm{Sch}}_{p,M}$ there are even guaranteed error bounds.
Conversely, the non-existence result of Theorem \ref{thm:main3} is also not obvious: we prove that \emph{regardless} of what operations are allowed, as long as the algorithm can only read a finite amount of information at each iteration, there will \emph{necessarily} be a potential for which the computation will fail.
In the spirit of the \emph{Solvability Complexity Theory} (see a brief discussion below in Subsection \ref{subsec:sci}, with more details in Section \ref{sec:sci}) we show that if one allows for two successive limits (which cannot be collapsed to a single limit), the computation is possible.
Note that although the class of potentials $\Omega^{\mathrm{Sch}}_{x_0}$ allows for a blowup near $x_0$, all potentials in this class are integrable over compact sets, and so $x_0$ is still a regular point
for the differential equation \cite[footnote, p.67]{MR0069338}; elsewhere the potentials are even more well-behaved.

\subsection{The Solvability Complexity Index Hierarchy}
\label{subsec:sci}
Our exploration of the spectral computational problem -- for both discrete and Schr\"odinger operators -- continues a line of research initiated by Hansen in \cite{Hansen11} and then further expanded in \cite{Ben-Artzi2015a,AHS}. This sequence of papers established the so-called \emph{Solvability Complexity Index (SCI) Hierarchy}, which is a classification of the computational complexity of problems that cannot be computed in finite time, only approximated. Recent years have seen a flurry of activity in this direction. We point out \cite{Colbrook2019,Colbrook2019a} where some of the theory of spectral computations has been further developed; \cite{R19}  where this has been applied to certain classes of unbounded operators; \cite{Becker2020b} where solutions of PDEs were considered; \cite{BMR2020,Ben-Artzi2020a} where we considered resonance problems; and \cite{Colbrook2019c} where the authors give further examples of how to perform certain spectral computations with error bounds.

While  precise definitions  are provided below in Section \ref{sec:sci}, let us provide an informal overview. Computational problems can be classed as belonging to one (or more) of $\Delta_k,\,\Sigma_k,\,\Pi_k$ for $k\in\N$ as follows:
	\begin{itemize}
	\item[$\Delta_k:$]
	For $k\geq2$, $\Delta_k$ is the class of problems that require  at most $k-1$ successive limits to solve. We also say that these problem have an $\SCI$ value of at most $k-1$. Problems in $\Delta_1$ can be solved in one limit with known error bounds.
	\item[$\Sigma_k:$]
	For all $k\in\N$, $\Sigma_k\subset\Delta_{k+1}$ is the class of problems in $\Delta_{k+1}$ that can be approximated from ``below'' with known error bounds. 
	\item[$\Pi_k:$]
	For all $k\in\N$, $\Pi_k\subset\Delta_{k+1}$ is the class of problems in $\Delta_{k+1}$ that can be approximated from ``above'' with known error bounds.
	\end{itemize}
By an approximation from ``above''  (resp. ``below'') we mean that the output of the algorithm is a superset (resp. subset) of the object we are computing (this clearly requires that this object and its approximations  belong to a certain  topological space). It can also be shown that for $k\in\{1,2,3\}$ we have $\Delta_k=\Sigma_k\cap\Pi_k$.

\begin{figure}[H]
\centering
\begin{tikzpicture}
\draw (0,0) circle (1) (-0.3,.9)  node [text=black,above] {$\Sigma_k$}
      (1,0) circle (1) (1.3,.9)  node [text=black,above] {$\Pi_k$}
      (.5,0) circle (1.95) (1.5,1.9) node [text=black,above] {$\Delta_{k+1}=\{\SCI\leq k\}$}
       (.5,-0.2) node [text=black,above] {$\Delta_{k}$}
       (4.3,0.4) node [text=black,above] {$=\{\SCI\leq k-1\}$};
\draw[->, dash dot] (2.9,0.7) -- (.6,0.4);
\end{tikzpicture}
\caption{The $\SCI$ Hierarchy for $k\in\{1,2,3\}$.}
\label{fig:sci-hir}
\end{figure}
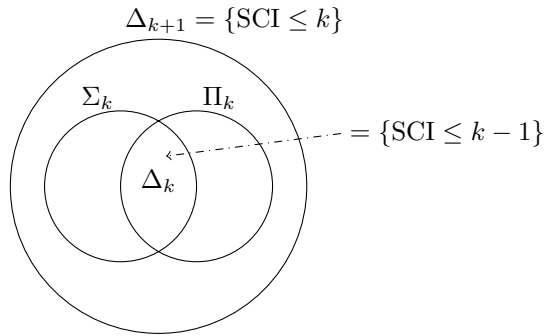

In \cite{AHS} the spectral computational problems for both $B(\ell^2(\N))$ and for Schr\"odinger operators were addressed. Some of the results shown there include\\

	\begin{tabularx}{\textwidth}{ll}
	approximating $\sigma(A)$ for $A\in B(\ell^2(\N))$&$\in\Pi_3\setminus\Delta_3$\\
	approximating $\sigma(A)$ for $A\in \{B\in B(\ell^2(\N))\ |\ B\text{ is selfadjoint}\}$&$\in\Sigma_2\setminus\Delta_2$\\
	approximating $\sigma(A)$ for $A\in \{B\in B(\ell^2(\N))\ |\ B\text{ is banded}\}$&$\in\Pi_2\setminus\Delta_2$\\
	approximating $\sigma(-\Delta+V)$ for $V$ bounded with known BV bounds\footnote{These known BV bounds mean that  for any $R>0$ one has \emph{a priori} knowledge of the total variation of $V$ on the ball $B_R$ of radius $R$ centered at the origin.} &$\in\Pi_2\setminus\Delta_1$.
	\end{tabularx}
	~\\

\medskip

\noindent Theorems  \ref{thm:main1},  \ref{th:main2} and  \ref{thm:main3} respectively imply that\\

	\begin{tabular}{ll}
	approximating $\sigma(A)$ for $A\in\Omega^{\mathrm{per}}$&$\in\Delta_2$\\
	approximating $\sigma(A)$ for $A\in\Omega^{\mathrm{per}}_{N,b}$&$\in\Delta_1$\\
	approximating $\sigma(-\Delta+V)$ for $V\in\Omega^{\mathrm{Sch}}$ &$\in\Delta_2$\\
	approximating $\sigma(-\Delta+V)$ for $V\in\Omega^{\mathrm{Sch}}_{p,M}$ &$\in\Pi_1$\\
	approximating $\sigma(-\Delta+V)$ for $V\in \Omega^{\mathrm{Sch}}_{x_0}$&$\in\Delta_3\setminus\Delta_2$.\\
	\end{tabular}

\bigskip

\noindent We point out that showing that a problem belongs to $\Pi_1,\,\Sigma_1$ or $\Delta_1$ is significant, as it shows that the computation can be done with certain guaranteed error bounds.

\subsection{Periodic Operators}
Schr\"{o}dinger equations with periodic potentials have been the subject of study since the earliest days of quantum mechanics, perhaps most famously
for the Bethe-Sommerfeld Conjecture \cite{BS-Conj}, which states that the number of gaps in the essential spectrum is finite in dimensions $\geq 2$. After
more than seventy years this conjecture was finally proved in complete generality for Schr\"{o}dinger operators in ${\mathbb R}^d$ by Parnovski \cite{MR2419769}.
In dimensions 2 and 3 the conjecture had already been proved by Popov and Skriganov \cite{MR629118} and Skriganov \cite{MR784531} respectively, and in
dimension 4 by Helffer and Mohamed \cite{MR1609321}. 

Beyond these results in mathematical analysis, since the 1990s interest in periodic problems has grown rapidly  in the 
applied analysis and computational mathematics literature, partly driven by models of photonic crystals. These models are typically based on time-harmonic 
Maxwell equations or upon second order elliptic equations with periodic coefficients. Figotin and Kuchment \cite{MR1257824,MR1372891,MR1417473} give 
particularly thorough analyses of some of these models,  showing that already in these cases with piecewise constant coefficients the associated operators may 
possess an arbitrarily large number of spectral gaps. 

For a periodic problem with some particular coefficients, often the first question of interest is whether it has any spectral gaps at all. Numerical methods
may be used to obtain some preliminary evidence, and are almost always based on the Floquet-Bloch decomposition. The fact that the coefficients are 
often only piecewise continuous requires substantial effort to be given to adaptive meshing, see Giani and Graham \cite{MR2909914}, although the continuous 
variation of the quasi-momentum over the Brillouin zone means that some of the effort can be recycled from one quasi-momentum to the next. Despite the substantial
computational cost, Floquet-Bloch techniques can sometimes be used to go beyond preliminary evidence, and have been combined with interval arithmetic to 
obtain algorithms which yield computer-assisted proofs of existence of spectral gaps for a wide class of problems  with coefficients expressed in terms
of elementary functions,  see Hoang, Plum and Wieners \cite{MR2556599}.
The literature for periodic problems which are not self-adjoint is much less extensive. For the ODE case, the work of Rofe-Beketov \cite{MR0157033} is generally
the starting point for any research on this topic.

	The discrete case encompasses numerous different directions of research. Typical examples are Toeplitz and Laurent operators \cite{Gohberg2003} or Jacobi operators  \cite{teschl2000jacobi}. Another direction that has gained a great deal of attention in the last two decades is that of periodic discrete Schr\"odinger operators and generalizations thereof. A particular type known as \emph{almost Mathieu operator} has been shown to exhibit rich spectral behavior (e.g. the spectrum can be a set of non-integer Hausdorff dimension, cf. \cite{avila2009ten, last2005spectral}). 
Beyond this, the literature pertaining to operators that are either not tri-diagonal or not self-adjoint is limited. The best starting point would be the book of Trefethen and Embree \cite{TE}.

\subsection*{Organization of the paper}
In Section \ref{sec:sci} we give a brief introduction to the main ideas of the $\SCI$ theory. Sections \ref{sec:matrices}, \ref{sec:schrodinger} and \ref{sec:counterexample} are dedicated to the proofs of Theorems \ref{thm:main1}, \ref{th:main2} and \ref{thm:main3}, respectively. Finally, in Section \ref{sec:numerics} we provide some numerical examples.

\section{The Solvability Complexity Index Hierarchy}
\label{sec:sci}
The Solvability Complexity Index ($\SCI$)  and the $\SCI$ Hierarchy provide a unified approach for understanding just how ``difficult'' it is to approximate infinite-dimensional problems (such as computing spectra) starting from finite-dimensional approximations.
 We start by setting the scene with a concrete example before providing precise abstract definitions.

\subsection{Informal discussion \& examples}
Consider the set $\Omega= B(\ell^2(\N))$ of all bounded operators on $\ell^2(\N)$. Let $\{e_i\}_{i\in\N}$ be the canonical basis. Then any element $A\in\Omega$ is represented by an infinite matrix. Denote {by} $\Lambda$ the set of all entries in this matrix. Then one could ask:

\emph{For any $A\in\Omega$, is it possible to compute its spectrum $\sigma(A)$  as the limit of a sequence of computations $\Gamma_n$, where each $\Gamma_n$ has access to only finitely many elements of $\Lambda$ and can only perform finitely many arithmetic computations?}

Needless to say, the whole point here is that the algorithms $\Gamma_n$ are not tailored for this specific element $A$: they are meant to be able to handle \emph{any} element $A\in\Omega$. The convergence of $\Gamma_n(A)$ to $\sigma(A)$ is made precise by realizing them as elements of the metric space $\mathcal M=(\mathrm{cl}(\C), d)$ which comprises all closed subsets of $\C$ endowed with an appropriate metric $d$ (such as the Hausdorff metric).

In \cite{Hansen11}, Hansen showed that it \emph{is} possible to compute $\sigma(A)$ for any $A\in\Omega$ as above. However, rather than having algorithms $\Gamma_n$ with a single index $n\in\N$,   \emph{three} indices were required, satisfying $\sigma(A)=\lim_{n_3\to\infty}\lim_{n_2\to\infty}\lim_{n_1\to\infty}\Gamma_{n_1,n_2,n_3}(A)$. The algorithms $\Gamma_{n_1,n_2,n_3}$ are given explicitly, and can be implemented numerically (though this raises a philosophical question about what it means to take successive limits numerically).  In \cite{AHS} it was proved that this is optimal: this computation cannot be performed with fewer than $3$ limits, and hence we say that this problem has an $\SCI$ value of $3$.

The $\SCI$ value strongly  depends  on $\Omega$: intuitively, if $\Omega$ contains fewer elements, then devising a `one-size-fits-all' algorithm should be easier. Indeed,  if one considers $\Omega^{\mathrm{sa}}:=\{A\in\Omega\,|\,A\text{ is selfadjoint}\}$ then the $\SCI$ value reduces to $2$ and for $\Omega^{\mathrm{cpt}}:=\{A\in\Omega\,|\,A\text{ is compact}\}$ it further reduces to $1$.

The classification into $\SCI$ values can be further refined into a classification that takes into account error bounds. This is the so-called \emph{SCI Hierarchy} which we describe below.

\subsection{Definitions}
We formalize the foregoing example with precise definitions:
\begin{de}[Computational problem]\label{def:computational_problem}
	A \emph{computational problem} is a quadruple $(\Om,\Lambda,\Xi,\mathcal M)$, where 
	\begin{enumi}
		\item $\Om$ is a set, called the \emph{primary set},
		\item $\Lambda$ is a set of complex-valued functions on $\Om$, called the \emph{evaluation set},
		\item $\mathcal M$ is a metric space,
		\item $\Xi:\Om\to \mathcal M$ is a map, called the \emph{problem function}.
	\end{enumi}
\end{de}
\begin{remark}
In this paper it is often clear what $\mathcal{M}, \Lambda$ and $\Xi$ are, and the important element is the primary set $\Omega$. In this case we may abuse notation and refer to $\Omega$ alone as the computational problem.
\end{remark}
\begin{de}[Arithmetic algorithm]\label{def:Algorithm}
	Let $(\Om,\Lambda,\Xi,\mathcal M)$ be a computational problem. An \emph{arithmetic algorithm} is a map $\Gamma:\Om\to\mathcal M$ such that for each $T\in\Om$ there exists a finite subset $\Lambda_\Gamma(T)\subset\Lambda$ such that
	\begin{enumi}
		\item the action of $\Gamma$ on $T$ depends only on $\{f(T)\}_{f\in\Lambda_\Gamma(T)}$,
		\item for every $S\in\Om$ with $f(T)=f(S)$ for all $f\in\Lambda_\Gamma(T)$ one has $\Lambda_\Gamma(S)=\Lambda_\Gamma(T)$,
		\item the action of $\Gamma$ on $T$ consists of performing only finitely many arithmetic operations on $\{f(T)\}_{f\in\Lambda_\Gamma(T)}$.
	\end{enumi}
\end{de}
\begin{de}[Tower of arithmetic algorithms]\label{def:Tower}
	Let $(\Om,\Lambda,\Xi,\mathcal M)$ be a computational problem. A \emph{tower of algorithms} of height $k$ for $\Xi$ is a family $\Gamma_{n_1,n_2,\dots,n_k}:\Om\to\mathcal M$ of arithmetic algorithms such that for all $T\in\Om$
	\begin{align*}
		\Xi(T) = \lim_{n_k\to\infty}\cdots\lim_{n_1\to\infty}\Gamma_{n_1,n_2,\dots,n_k}(T).
	\end{align*}
\end{de}
\begin{de}[SCI]
	A computational problem $(\Om,\Lambda,\Xi,\mathcal M)$ is said to have a \emph{Solvability Complexity Index} ($\SCI$) of $k$ if $k$ is the smallest integer for which there exists a tower of algorithms of height $k$ for $\Xi$.
	If a computational problem has solvability complexity index $k$, we write \begin{align*}
 			\SCI(\Om,\Lambda,\Xi,\mathcal M)=k.
		 \end{align*}
If there exists a family $\{\Gamma_n\}_{n\in\N}$ of arithmetic algorithms and $N_1\in\N$ such that $\Xi=\Gamma_{N_1}$ then we define $\SCI(\Om,\Lambda,\Xi,\mathcal M)=0$.
\end{de}

\begin{de}[The SCI Hierarchy]
\label{1st_SCI}
The \emph{$\SCI$ Hierarchy} is a hierarchy $\{\Delta_k\}_{k\in{\N_0}}$ of classes of computational problems $(\Om,\Lambda,\Xi,\mathcal M)$, where each $\Delta_k$ is defined as the collection of all computational problems satisfying:
\begin{align*}
(\Om,\Lambda,\Xi,\mathcal M)\in\Delta_0\quad &\qquad\Longleftrightarrow\qquad \mathrm{SCI}(\Om,\Lambda,\Xi,\mathcal M)= 0,\\
(\Om,\Lambda,\Xi,\mathcal M)\in\Delta_{k+1} &\qquad\Longleftrightarrow\qquad \mathrm{SCI}(\Om,\Lambda,\Xi,\mathcal M)\leq k,\qquad k\in\N,
\end{align*}
with the special class $\Delta_1$  defined as the class of all computational problems in $\Delta_2$ with a convergence rate:
\begin{equation*}
(\Om,\Lambda,\Xi,\mathcal M)\in\Delta_{1} \qquad\Longleftrightarrow\qquad
 \exists  \{\Gamma_n\}_{n\in \mathbb{N}},\,\exists\epsilon_n\downarrow0 \quad\text{ s.t. }\quad \forall  T\in\Omega, \ d(\Gamma_n(T),\Xi(T)) \leq \epsilon_n.
\end{equation*}
Hence we have that $\Delta_0\subset\Delta_1\subset\Delta_2\subset\cdots$
\end{de}

When the metric space $\mathcal{M}$ has certain ordering properties, one can define further classes that take into account convergence from below/above and associated error bounds. In order to not burden the reader with unnecessary definitions, we provide the definition that is relevant to the cases where $\mathcal{M}$ is the space of closed (and bounded) subsets of $\R^d$ together with the Attouch-Wets (Hausdorff) distance (definitions of which can be found in Appendix \ref{sec:metrics}). These are the cases of relevance to us. A more comprehensive and abstract definition can be found in \cite{AHS}.

\begin{de}[The SCI Hierarchy (Attouch-Wets/Hausdorff metric)]
\label{def:pi-sigma}
Consider the setup in Definition \ref{1st_SCI} assuming further that $\mathcal{M}=(\mathrm{cl}(\R^d),d)$ where $d=d_{\mathrm{AW}}$ or $d=d_{\mathrm{H}}$.   Then for $k\in\N$ we can define the following subsets of $\Delta_{k+1}$:
\begin{equation*}
	\begin{split}
	\Sigma_{k}
	&=
	\Big\{(\Om,\Lambda,\Xi,\mathcal M) \in \Delta_{k+1} \ \vert \  \exists\{ \Gamma_{n_1,\dots,n_k}\}\text{ s.t. }    \forall T \in \Omega,\, \exists \{X_{n_k}(T)\}\subset\mathcal{M}, \text{ s.t. } \\
	&\qquad\lim_{n_k\to\infty}\cdots\lim_{n_1\to\infty}\Gamma_{n_1,\dots,n_k}(T)=\Xi(T)  \text{ \& } \Gamma_{n_1,\dots,n_k}(T)\subset X_{n_k}(T)\text{ \& }d\left(X_{n_k}(T),\Xi(T)\right)\leq \epsilon_{n_k} \Big\},
	\\
	\Pi_{k}
	&=
	\Big\{(\Om,\Lambda,\Xi,\mathcal M) \in \Delta_{k+1} \ \vert \  \exists\{ \Gamma_{n_1,\dots,n_k}\}\text{ s.t. }    \forall T \in \Omega,\, \exists \{X_{n_k}(T)\}\subset\mathcal{M}, \text{ s.t. } \\
	&\qquad\qquad\lim_{n_k\to\infty}\cdots\lim_{n_1\to\infty}\Gamma_{n_1,\dots,n_k}(T)=\Xi(T)  \text{ \& } \Xi(T)\subset X_{n_k}(T) \\
	&\qquad\qquad \qquad\qquad\qquad\qquad\text{ \& }d\Bigl(X_{n_k}(T),\lim_{n_{k-1}\to\infty}\cdots\lim_{n_1\to\infty}\Gamma_{n_1,\dots,n_k}(T)\Bigr)\leq \epsilon_{n_k} \Big\}.
	\end{split}
\end{equation*}
It can be shown that $\Delta_k=\Sigma_k\cap\Pi_k$ for $k\in\{1,2,3\}$, see Figure \ref{fig:sci-hir}. We  refer to \cite{AHS} for a detailed treatise.
\end{de}

\section{Banded Periodic  Matrices}
\label{sec:matrices}

In this section we prove Theorem \ref{thm:main1}. We do this by defining an explicit algorithm  and show that its output converges to the desired spectrum. The two computational problems we consider only differ in the primary set $\Om$, and are as follows:

\begin{align}\label{eq:periodic_computational_problem}
\left\{
\begin{array}{rl}
	\hfill \Omega &=\quad \Omega^{\mathrm{per}}\quad\text{or}\quad\Omega^{\mathrm{per}}_{N,b}\\[1mm]
	\mathcal M &=\quad \big(\{K\subset\C\,|\, K \text{ compact}\}, \, d_{\mathrm{H}}\big)\\[1mm]
	\Lambda &=\quad \{\Omega\ni A\mapsto \langle e_i, Ae_j\rangle\,|\, i,j\in\Z\}\\[2mm]
	\Xi &: \quad\Om \to \mathcal M;\quad
	A\mapsto \sigma(A),
\end{array}
\right.
\end{align}
 where $\{e_i\}_{i\in\Z}$ denotes the canonical basis of $\ell^2(\Z)$ and  $d_\mathrm{H}$ denotes the Hausdorff distance. We remind the reader  that  $\Omega^{\mathrm{per}}_{N,b}$ is the class of operators on $\ell^2(\Z)$ whose canonical matrix representation has bandwidth  $b$ (i.e. $A_{ij}=0\,\,\forall |i-j|>b$) and whose matrix entries repeat periodically along the diagonals with period $N$ (here $N,b\in\N$). Clearly, every $A\in\Omega^{\mathrm{per}}_{N,b}$ defines a bounded operator on $\ell^2(\Z)$. Note that $\Omega^{\mathrm{per}} = \bigcup_{N,b\in\N}\Omega^{\mathrm{per}}_{N,b}$ is the class of operators on $\ell^2(\Z)$ whose canonical matrix representation is banded and whose matrix entries repeat periodically along the diagonals. In the language of the Solvability Complexity Index, the two parts of Theorem \ref{thm:main1} can be expressed as follows:

\begin{itemize}
\item
Part (i) amounts to proving  that the computational problem for $\Om^{\mathrm{per}}$ has an $\SCI$ value of $1$ (or, equivalently,   it belongs to $\Delta_2$). 
\item
Part (ii) amounts to showing that the computational problem for  $\Omega^{\mathrm{per}}_{N,b}$ belongs to $\Delta_1$, i.e. it can be approximated with explicit error bounds; this is restated as Theorem \ref{th:matrix_mainth} below.
\end{itemize}

\begin{remark}
We note that the Hausdorff distance is only defined for non-empty sets, and it is finite only if the sets are bounded. Hence it is important to observe that for any $A\in\Omega^{\mathrm{per}}$, the set $\sigma(A)$ is both non-empty and bounded. Indeed, boundedness of the spectrum follows immediately from  boundedness of $A$, while non-emptyness follows from the Floquet-Bloch theory described in Section \ref{sec:Floquet_Bloch}.
We discuss the metrics used in this paper in Appendix \ref{sec:metrics}.
\end{remark}
\begin{example}\label{ex:tridiag}
	The class $\Om_{N,1}^{\mathrm{per}}$ contains all Jacobi-type matrices of the form
	\begin{align}\label{eq:A_tridiag}
		A = \begin{pmatrix}
			\ddots & \ddots & \ddots & & & & \\
			& c_0 & a_0 & b_0 & & & 0& & \\
			& & c_1 & a_1 & b_1 & & & & \\
			& & & \ddots & \ddots & \ddots & & & \\
			& & & & c_{N-1} & a_{N-1} & b_{N-1} & & \\
			& &0 & & & c_0 & a_0 & b_0 & \\
			& & & & & & \ddots & \ddots & \ddots 
		\end{pmatrix} \in \Om_{N,1}^{\mathrm{per}}.
	\end{align}
\end{example} 
\subsection{Proof of Theorem \ref{thm:main1}{(i)}}
\label{subsec:proof-main-thm1}
To prove Theorem \ref{thm:main1}(i) we assume to be known    Theorem \ref{thm:main1}(ii). Theorem \ref{thm:main1}(ii)  can be restated in the language of the SCI Hierarchy as follows:
 \begin{theorem}\label{th:matrix_mainth}
 	For any fixed $N,b\in\N$ the computational problem for $\Omega^{\mathrm{per}}_{N,b}$ can be solved in one limit with explicit error bounds, i.e. $\big(\Omega^{\mathrm{per}}_{N,b},\Lambda,\Xi,\mathcal M\big) \in \Delta_1$.
 \end{theorem}

The proof of this theorem is contained in Subsection \ref{subsec:proof} below, after some preparatory work. First, however, we prove  Theorem \ref{thm:main1}{(i)}:

\begin{proof}[Proof of Theorem \ref{thm:main1}{(i)}]
	By Theorem \ref{th:matrix_mainth}, for every $N\in\N$ there exists a family of algorithms $\{\Gamma^{(N)}_n\}_{n\in\N}$, such that $\Gamma^{(N)}_n(B)\to\sigma(B)$ as $n\to+\infty$ for any $B\in\Omega^{\mathrm{per}}_{N,b}$. Now, let $A=(a_{ij})_{i,j\in\Z}\in\Omega^{\mathrm{per}}$ and define a new family $\{\Gamma_n\}_{n\in\N}$ by the following pseudocode. 
	
	\medskip
	\begin{algorithm}[H]\label{eq:pseudocode}
	\DontPrintSemicolon
	\SetAlgorithmName{Pseudocode}{}{}
	 \For{$n\in\N$}{
		 For $i\in\{-n,\dots,n\}$ define $d_i := (a_{i,i-n}, a_{i,i-n+1}, \dots, a_{i,i}, \dots, a_{i,i+n-1}, a_{i,i+n})$\;
		  \eIf{$\exists p<n$ \textnormal{s.t.} $d_{i+p} = d_i \; \forall i\in \{p,\dots,n-p\}$}{
			   $N:=\min\{p \,|\, d_{\bullet+p} = d_{\bullet}\}$\;
		  }{
		  $N:=n$
		  }
	 Define $B_n:=(b_{ij})_{i,j\in\Z}, \text{ where } \begin{cases}
							(b_{i,i-n}, \dots, b_{i,i}, \dots, b_{i,i+n}) := d_{(i\text{ mod }N)} &\text{for } i\in\Z \\
							 \hfill b_{ij}:=0 \phantom{omodNo} & \text{otherwise}
				 \end{cases}$\;
	 Define $\Gamma_n(A) := \Gamma_n^{(N)}(B_n)$\;
	 }
	 \caption{Definition of $\{\Gamma_n\}_{n\in\N}$ on $\Omega^{\mathrm{per}}$}
	 \label{alg:pseudocode}
	\end{algorithm}
	
	\medskip
	To clarify the meaning of $d_i$ we note that in Example \ref{ex:tridiag} one would have $d_i = (\dots,0,c_i,a_i,b_i,0,\dots)$.
	Loosely speaking, Pseudocode \ref{alg:pseudocode} first takes a finite section of $A$, searches it for periodic repetitions, and then defines a matrix $B_n\in\Omega^{\mathrm{per}}_{N,b}$ by periodic extension, to which $\Gamma^{(N)}_n$ can be applied. Because $A$ is banded and periodic, this routine will eventually find its period: there exists $n_0\in\N$ such that for all $n>n_0$,  $N$ (as defined in the routine) is equal to the period of $A$ and $B_n\equiv A$. Hence, for  $n>n_0$  we have $\Gamma_n(A) = \Gamma_n^{(N)}(B_n) = \Gamma_n^{(N)}(A)\to\sigma(A)$ as $n\to+\infty$, by the properties of $\Gamma_n^{(N)}$.
	
	Finally, note that every line of Pseudocode \ref{alg:pseudocode}  can be executed with finitely many algebraic operations on the matrix elements of $A$.
\end{proof}
The following subsections are devoted to the proof of Theorem \ref{th:matrix_mainth}.
 The proof is constructive, i.e. we will provide an explicit algorithm that computes the spectrum of any given operator $A\in \Omega^{\mathrm{per}}_{N,b}$ with explicit error bounds.

 \subsection{Floquet-Bloch transform}\label{sec:Floquet_Bloch}

 Let $N$ be as in the statement of Theorem \ref{th:matrix_mainth}. Given a vector $x=(x_n)_{n\in\Z}\in\ell^2(\Z)$ and given $\theta\in[0,2\pi]$, define
 \begin{align}\label{eq:U_theta_def}
 	(\cU_\theta x)_n := (2\pi)^{-\f12} \sum_{k\in\Z} x_{n+kN}e^{-\i\theta(k+\f n N)}.
 \end{align}
 We also introduce the symbol $\ell^2_{\text{per}}(N)$ to denote the space of all $N$-periodic sequences $(y_k)_{k\in\Z}$, together with the norm  $\|y\|_{\ell^2_{\text{per}}(N)}^2 = \sum_{k=0}^{N-1} |y_k|^2$. Note that $\ell^2_{\text{per}}(N)$  is canonically isomorphic to the Euclidean space $\R^N$.
 The following lemma is easily proved by direct computation.
 \begin{lemma}[Properties of $\cU_\theta$]
 	The map $\cU_\theta$ defined in \eqref{eq:U_theta_def} has the following properties.
 	\begin{enumi}
 		\item For any $x\in\ell^2(\Z)$, $\cU_\theta x$ is $N$-periodic, that is $\cU_\theta:\ell^2(\Z)\to \ell^2_{\text{per}}(N)$;
 		\item The map
 		\begin{align*}
 			\cU : \ell^2(\Z) &\to \int_{[0,2\pi]}^\oplus \ell^2_{\mathrm{per}}(N)\, d\theta \\
 			x &\mapsto (\cU_\theta x)_{\theta\in[0,2\pi]}
 		\end{align*}
 		is unitary;
 		\item The inverse $\cU^{-1}$ is given by 
 		\begin{align*}
 			(\cU^{-1}y)_n = (2\pi)^{-\f12} \int_0^{2\pi} y_n(\theta)e^{\i n \f{\theta}{N}} \,d\theta.
 		\end{align*}
 	\end{enumi}
 \end{lemma}
\begin{proof}
This is standard, and the proof is omitted.
\end{proof}

\subsection{Transform and Properties of $A\in\Omega^{\mathrm{per}}_{N,b}$}

The $N$-periodicity of $A\in\Omega^{\mathrm{per}}_{N,b}$ along diagonals is equivalent to the identity
\begin{align}\label{eq:periodic_identity}
	A_{m,n} = A_{m+kN, n+kN} \qquad\text{for all }m,n,k\in\Z.
\end{align}
\begin{lemma}
	For any $A\in\Omega^{\mathrm{per}}_{N,b}$, $y\in \ell^2_{
	\mathrm{per}}(N)$ and $\theta\in[0,2\pi]$, define $(A(\theta)y)_n = \sum_{j\in\Z}e^{\i\theta\f{j-n}{N}} A_{nj}y_j$. Then  one has 
	\begin{align*}
		\cU A \cU^{-1} = \int_{[0,2\pi]}^\oplus A(\theta)\,d\theta.
	\end{align*}
\end{lemma}
\begin{proof}
	Let $x\in\ell^2(\Z)$. We have
	\begin{align*}
		(\cU_\theta Ax)_n &= (2\pi)^{-\f12} \sum_{k\in\Z} (Ax)_{n+kN} e^{-\i\theta(k+\f{n}{N})} \\
		&= (2\pi)^{-\f12} \sum_{k\in\Z}\sum_{j\in\Z} A_{n+kN,j} x_{j} e^{-\i\theta(k+\f{n}{N})} \\
		&= (2\pi)^{-\f12} \sum_{k\in\Z}\sum_{j\in\Z} A_{n,j-kN} x_{j} e^{-\i\theta(k+\f{n}{N})} \\
		&= (2\pi)^{-\f12} \sum_{k\in\Z}\sum_{j\in\Z} A_{nj} x_{j+kN} e^{-\i\theta(k+\f{n}{N})} \\
		&= (2\pi)^{-\f12} \sum_{k\in\Z}\sum_{j\in\Z} \left( e^{-\i\theta\f{n-j}{N}} A_{nj} \right) x_{j+kN} e^{-\i\theta(k+\f{j}{N})} \\
		&= \sum_{j\in\Z} \left( e^{-\i\theta\f{n-j}{N}} A_{nj} \right) (2\pi)^{-\f12} \sum_{k\in\Z} x_{j+kN} e^{-\i\theta(k+\f{j}{N})} \\
		&= (A(\theta)\cU_\theta x)_n,
	\end{align*}
	where we have used the periodicity of $A$ (see \eqref{eq:periodic_identity}) in the third line. The assertion now follows from the invertibility of $\cU$.
\end{proof}
\begin{remark}
Observe that $(A(\theta)y)_n=(A(\theta)y)_{n+kN}$ for any $k\in\Z$ so that $A(\theta)y\in \ell^2_{\mathrm{per}}(N)$. Hence, $A(\theta)$ is an operator $ \ell^2_{\mathrm{per}}(N)\to \ell^2_{\mathrm{per}}(N)$, which can be expressed as an $N\times N$ matrix. Note, however, that the numbers $e^{-\i\theta\f{n-j}{N}} A_{nj}$ are \emph{not} the matrix elements of $A(\theta)$ with respect to any basis. Indeed, $\ell^2_{\text{per}}(N)$ is finite-dimensional, while $e^{-\i\theta\f{n-j}{N}} A_{nj}$ ($n,j\in\Z$) are infinitely many numbers.
\end{remark} 
 As noted earlier, we have $\ell^2_{\text{per}}(N) \cong \R^N$. The identification can be made via the basis 
 \begin{align*}
 	e_1^\text{per} &= (\dots, 0, \underbrace{1, 0, \dots, 0}_{N\text{ entries}}, 1, 0, \dots) \\
 	e_2^\text{per} &= (\dots, 0, 0, \underbrace{1, 0, \dots, 0}_{N\text{ entries}}, 1\dots) \\
 	&\quad \vdots
 \end{align*}
i.e. $(e_n^{\text{per}})_j = \delta_{(j\text{ mod }N),\,n}$ for $n\in\{0,\dots,N-1\}$ (Kronecker symbol). In this basis, the matrix elements of $A(\theta)$ become
\begin{align}
	A(\theta)_{mn} &= \bigl\langle e_m^\text{per}, A(\theta) e_n^\text{per} \bigr\rangle_{\ell^2_{\text{per}}(N)} 
	\nonumber\\
	&= \sum_{k=0}^{N-1} \delta_{(k\text{ mod }N),\,m} \sum_{j\in\Z} e^{\i\theta\f{j-k}{N}} A_{kj} \delta_{(j\text{ mod }N),\,n}
	\nonumber\\
	&= \sum_{j\in\Z} \sum_{k=0}^{N-1} \delta_{k,m} e^{\i\theta\f{j-k}{N}} A_{kj} \delta_{(j\text{ mod }N),\,n}
	\nonumber\\
	&= \sum_{j\in\Z} e^{\i\theta\f{j-m}{N}} A_{mj} \delta_{(j\text{ mod }N),\,n}
	\nonumber\\
	&= \sum_{j'\in\Z} e^{\i\theta\f{j'N+n-m}{N}} A_{m,\, j'N+n} 
	\nonumber\\
	&= e^{\i\theta\f{n-m}{N}} \sum_{j'\in\Z} e^{\i\theta j'} A_{m,\, j'N+n}.
	\label{eq:A(theta)_formula}
\end{align}
Note that the sum in the last line is actually finite, because $A$ is banded. Indeed, if the band width of $A$ is less than the period, then the sum over $j'$ in \eqref{eq:A(theta)_formula} contains only one term.
\begin{example}
	If $A$ is a matrix with $N=1$, i.e. $A$ is a Laurent operator, then formula \eqref{eq:A(theta)_formula} yields a scalar function of $\theta$ given by
	\begin{align*}
		A(\theta) = \sum_{j\in\Z} e^{\i\theta j} A_{0, j}. 
	\end{align*}
	Writing $z:=e^{\i\theta}$, we see that $A(\theta) = \sum_{j\in\Z} z^j A_{0, j}$ is given by the \emph{symbol} of the Laurent operator. We thus recover the classical result that the spectrum of a Laurent operator is given by the image of the unit circle under its symbol (cf. \cite[Th. 7.1]{TE}).
\end{example}
\begin{example}
	If $A$ is tri-diagonal and $N=5$ (cf. \eqref{eq:A_tridiag}), the formula above gives
	 \begin{align*}
	 	A(\theta) = \begin{pmatrix}
	 		a_0 & b_0 e^{\i\f{\theta}{5}} & 0 & 0 & c_0 e^{-\i\f{\theta}{5}} \\
	 		c_1 e^{-\i\f{\theta}{5}} & a_1 & b_1 e^{\i\f{\theta}{5}} & 0 & 0\\
	 		0 & c_2 e^{-\i\f{\theta}{5}} & a_2 & b_2 e^{\i\f{\theta}{5}} & 0\\
	 		0 & 0 & c_3 e^{-\i\f{\theta}{5}} & a_3 & b_3 e^{\i\f{\theta}{5}}\\
	 		b_4 e^{\i\f{\theta}{5}} & 0 & 0 & c_4 e^{-\i\f{\theta}{5}} & a_4
	 	\end{pmatrix}.
	 \end{align*}
\end{example}

Next, we establish some elementary facts about the spectrum of a periodic operator.
 By standard results about the Floquet-Bloch transform, we have
 \begin{align*}
 	\sigma(A) = \bigcup_{\theta\in[0,2\pi]} \sigma(A(\theta))
 \end{align*}
 for all $A\in\Omega^{\mathrm{per}}_{N,b}$. Thus, an algorithm may be devised by determining the zeros of the map $z\mapsto \det(A(\theta)-zI)$, $\theta\in[0,2\pi]$. To this end, note that by definition of the determinant, one has
 \begin{align*}
 	\det(A(\theta)-zI) = \sum_{n=0}^{N} p_n(\theta)z^n,
 \end{align*}
 where the coefficient functions $p_n(\theta)$ are polynomials in the matrix entries $A(\theta)_{mn}$ and hence analytic and periodic in $\theta$.
 Hence they are bounded:
 \begin{align*}
 	\exists C>0\, :\, |p_n(\theta)|\leq C\qquad\forall\theta\in[0,2\pi],\,\,\forall n\in\{0,\dots,N\}.
 \end{align*}
 Moreover, $p_N(\theta)\equiv 1$.

\begin{lemma}\label{lemma:claim}
Let $A\in\Omega^{\mathrm{per}}_{N,b}$ and $R>0$. For any $z,w\in\C$ with $|z|,|w|\leq R$ and $\theta,\vartheta\in[0,2\pi]$ one has
 	\begin{equation*}\label{eq:Ctilde}
	\begin{split}
 		|\det(A(\theta)-zI)-\det(A(\vartheta)-wI)|& \leq \\
		N^{\f{N}{2}+2}\big((2b+1)\|A\|_{\infty}& + R\big)^N (2b+1)^2 \|A\|_{\infty} (|z-w|+|\theta-\vartheta|),
	\end{split}
 	\end{equation*}
 	where we note that $\|A\|_{\infty}=\max\{|A_{ij}|\,|\,i,j\in\Z\}$ can be computed in finitely many steps.
\end{lemma}

 \begin{proof}
	Denote $B_R:=\{z\in\C\,|\,|z|\leq R\}$.
	From the mean value theorem it follows that for any differentiable function $f:B_R\times[0,2\pi]\to\C$ one has
	\begin{align*}
		|f(z,\theta)-f(w,\vartheta)| \leq \|\nabla f\|_{L^\infty(B_R\times[0,2\pi])}\big( |z-w|+|\theta-\vartheta| \big).
	\end{align*}
	Hence to prove the claim it is enough to bound $\|\nabla f\|_{L^\infty(B_R)\times[0,2\pi]}$ for $f(z,\theta) = \det(A(\theta)-zI)$. This follows from the Jacobi formula: for any square matrix $M$ one has
	\begin{align*}
		\f{\del\det(M)}{\del M_{ij}} &= \cof(M)_{ij},
	\end{align*}
	where $\cof(M)$ denotes the cofactor matrix of $M$. Hence,
	\begin{align*}
		\f{\del}{\del z}\det(A(\theta)-zI) &= \sum_{i,j=0}^{N-1} \cof(A(\theta)-zI)_{ij} (-\delta_{ij}),
		\\
		\f{\del}{\del \theta}\det(A(\theta)-zI) &= \sum_{i,j=0}^{N-1} \cof(A(\theta)-zI)_{ij} \f{\del A(\theta)_{ij}}{\del \theta}.
	\end{align*}
	Using Hadamard's inequality to bound the cofactor matrix, we obtain the bounds
	\begin{align}
		\Big|\f{\del}{\del z}\det(A(\theta)-zI)\Big| &\leq N^{\f{N}{2}+1}\|A(\theta)-zI\|_{\infty}^N
		\nonumber
		\\
		&\leq N^{\f{N}{2}+1}\big((2b+1)\|A\|_{\infty} + R\big)^N,
		\label{eq:dz_bound}
		\\
		\Big|\f{\del}{\del \theta}\det(A(\theta)-zI)\Big| &\leq N^{\f{N}{2}+2}\|A(\theta)-zI\|_{\infty}^N \|\del_\theta A(\theta)\|_{\infty}
		\nonumber
		\\
		&\leq N^{\f{N}{2}+2}\big((2b+1)\|A\|_{\infty} + R\big)^N \|\del_\theta A(\theta)\|_{\infty}
		\nonumber
		\\
		&\leq N^{\f{N}{2}+2}\big((2b+1)\|A\|_{\infty} + R\big)^N (2b+1)^2 \|A\|_{\infty},
		\label{eq:dtheta_bound}
	\end{align}
	where the last two lines follow from the explicit formula \eqref{eq:A(theta)_formula}. The bounds \eqref{eq:dz_bound} and \eqref{eq:dtheta_bound} imply
	\begin{align*}
		\max\{|\del_\theta f|,|\del_z f|\} &\leq N^{\f{N}{2}+2}\big((2b+1)\|A\|_{\infty} + R\big)^N (2b+1)^2 \|A\|_{\infty}
	\end{align*}
	and the claim follows.
 \end{proof}

 \subsection{Proof of Theorem \ref{thm:main1}(ii)}
 \label{subsec:proof}

 We can finally prove Theorem \ref{thm:main1}(ii) which was restated equivalently as Theorem \ref{th:matrix_mainth}. First, we  define the family of algorithms $\{\Gamma_n\}_{n\in\N}$, where each of them maps $\Gamma_n:\Omega^{\mathrm{per}}_{N,b}\to\mathcal{M}$ (we recall that $\mathcal{M}$ is the space of all compact subsets of $\C$ endowed with the Hausdorff metric). 
 It is easy to see that for any $A\in\Omega^{\mathrm{per}}_{N,b}$ one has $\|A\|_{\ell^2\to\ell^2}\leq R_A:=\sum_{j=1}^N\sum_{k=-b}^b |A_{jk}|$ (this follows from Young's inequality), a quantity which can be computed in finitely many steps. Therefore, if we denote $B_{R_A}:=\{z\in\C\,|\,|z|\leq R_A\}$, the \emph{a priori} inclusion $\sigma(A)\subset B_{R_A}$ holds true for any $A\in\Omega^{\mathrm{per}}_{N,b}$.
 
\begin{de}[$N,b$-Periodic Matrix Algorithm]\label{def:periodic_matrix_alg}
	Let $A\in\Omega^{\mathrm{per}}_{N,b}$ and for $n\in\N$, let $\Theta_n = \big(\theta_1^{(n)},\dots,\theta_n^{(n)}\big)$ be a linear spacing of $[0,2\pi]$ and let $L_n:=\f1n(\Z+\i\Z) \cap B_{R_A}$  be a finite lattice with spacing $n^{-1}$. Then we define
	\begin{align}\label{eq:periodic_matrix_alg}
		\Gamma_n(A) := \bigcup_{i=1}^n \left\{ z\in L_n\,\middle|\, \big|\det\big(A(\theta_i^{(n)})-zI\big)\big|\leq n^{-\f12} \right\}.
	\end{align}
\end{de} 

\begin{remark}
We emphasize that \eqref{eq:periodic_matrix_alg} can be computed in finitely many arithmetic operations on the matrix elements of $A$. Indeed, computing the radius $R_A$ consists of a finite sequence of multiplications and additions, as does the computation of each of the determinants $\det(A(\theta_i^{\scriptscriptstyle(n\scriptscriptstyle)})-z)$ for $z$, $\theta_i^{\scriptscriptstyle(n\scriptscriptstyle)}$ in the finite sets $L_n$, $\Theta_n$.
\end{remark}
 \begin{proof}[Proof of Theorem \ref{th:matrix_mainth} (equiv. Theorem \ref{thm:main1}(ii))]
	The proof has two steps.\\
	
	\noindent\underline{{\bf Step 1}:  $\sigma(A)$ is approximated from above by $\Gamma_n(A)$.}
For any set $K\subset\C$ we denote by $B_\delta(K)$ the $\delta$-neighborhood of $K$.
	
	Let $z\in\sigma(A)$. Then $|z|\leq \|A\|_{\ell^2\to\ell^2} \leq R_A$ and there exists $\theta\in[0,2\pi]$ such that $\det(A(\theta)-zI)=0$. Choose $z_n\in L_n$ such that $|z-z_n| \leq n^{-1}$ and $\theta_n\in\Theta_n$ such that $|\theta-\theta_n| \leq n^{-1}$. Applying Lemma \ref{lemma:claim} we obtain the bound
	\begin{align*}
		|\det(A(\theta_n)-z_nI)| &\leq N^{\f{N}{2}+2}\big((2b+1)\|A\|_{\infty} + R_A\big)^N (2b+1)^2 \|A\|_{\infty} (|z-z_n|+|\theta-\theta_n|)
		\\
		&\leq \f2n N^{\f{N}{2}+2}\big((2b+1)\|A\|_{\infty} + R_A\big)^N (2b+1)^2 \|A\|_{\infty} .
	\end{align*}
	This inequality implies that $z_n\in\Gamma_n(A)$ as soon as $2n^{-1} N^{\f{N}{2}+2}\big((2b+1)\|A\|_{\infty} + R_A\big)^N (2b+1)^2 \|A\|_{\infty} \leq n^{-\f12}$, or equivalently, 
	\begin{align}\label{eq:quant_matrix_bound_1}
		n > \left(2 N^{\f{N}{2}+2}\big((2b+1)\|A\|_{\infty} + R_A\big)^N (2b+1)^2 \|A\|_{\infty}\right)^2.
	\end{align}
	Note that the right-hand side of \eqref{eq:quant_matrix_bound_1} is computable in finitely many arithmetic operations if $N$ and $b$ are known \emph{a priori}. Since $|z-z_n| \leq n^{-1}$ by construction, this shows that $\sigma(A)\subset B_{\f1n}(\Gamma_n(A))$ for $n > (2 N^{\f{N}{2}+2}\big((2b+1)\|A\|_{\infty} + R_A\big)^N (2b+1)^2 \|A\|_{\infty})^2$.
	
	\medskip
	\noindent\underline{{\bf Step 2}:  $\sigma(A)$ is approximated from below by $\Gamma_n(A)$.}
	Next we prove   that $\Gamma_n(A)\subset B_{\delta}(\sigma(A))$ for $n$ large enough. We first note that, since $\det(zI-A(\theta))$ is a polynomial in $z$, it can be factored to take the form
	\begin{align}\label{eq:factorization}
		\det(zI-A(\theta)) = \prod_{i=1}^N (z-z_i(\theta)),
	\end{align}
	where $z_i(\theta)$ are the zeros of $z\mapsto \det(zI-A(\theta))$ (note that for a characteristic polynomial the coefficient of the leading order term is always 1). From \eqref{eq:factorization} we  obtain the bound 
	\begin{equation}
		|\det(zI-A(\theta))| = \prod_{i=1}^N |z-z_i(\theta)|
		\geq \dist(z,\sigma(A))^N.
		\label{eq:dist_bound}
	\end{equation}
	Let $z_n\in\Gamma_n(A)$ be an arbitrary sequence. Then, by definition, $|\det(z_nI-A(\theta_n))|\leq n^{-\f12}$ for some $\theta_n\in\Theta_n$. From \eqref{eq:dist_bound} we conclude that
	\begin{align*}
		\dist(z_n,\sigma(A))^N \leq |\det(z_nI-A(\theta_n))|
		\leq n^{-\f12}
	\end{align*}
	and thus $z_n\in B_{n^{-\nf{1}{2N}}}(\sigma(A))$ for all $n\in\N$. This concludes step 2.
	
	\medskip
	Together, steps 1 and 2 imply that for any given $\delta>0$ one has both $\Gamma_n(A)\subset B_\delta(\sigma(A))$ and $\sigma(A)\subset B_\delta(\Gamma_n(A))$ provided that
	\begin{align}\label{eq:quantitative_n_bound}
		n > \max\left\{ \delta^{-2N},\, \left(2 N^{\f{N}{2}+2}\big((2b+1)\|A\|_{\infty} + R_A\big)^N (2b+1)^2 \|A\|_{\infty}\right)^2 \right\}.
	\end{align}
	Since the right-hand side of \eqref{eq:quantitative_n_bound} is computable from $N$, $b$ and the matrix elements of $A$ in finitely many arithmetic operations, we conclude that the computational problem is in $\Delta_1$.
 \end{proof}
\begin{example}\label{example:matrices}
	The algorithm from Definition \ref{def:periodic_matrix_alg} can easily be implemented in Matlab.  An example calculation with bandwidth $b=1, N=5$ and $A$ of the form \eqref{eq:A_tridiag} with
	\begin{align*}
		(a_i) &= (1, 0, 1, 0, 2)
		\\
		(b_i) &= (-1, -2, 1, 3\i, -5)
		\\
		(c_i) &= (2\i, -3\i, 2\i, 0, \i)
	\end{align*}
	yields the following output in the complex plane. Note that the output is a set of points on a discrete grid in the complex plane and hence the output looks `fat'. As $n$ is taken larger this set of points dwindles to just those points that lie in an ever decreasing neighborhood of the true spectrum.
	\begin{figure}[H]
		\centering
		\includegraphics{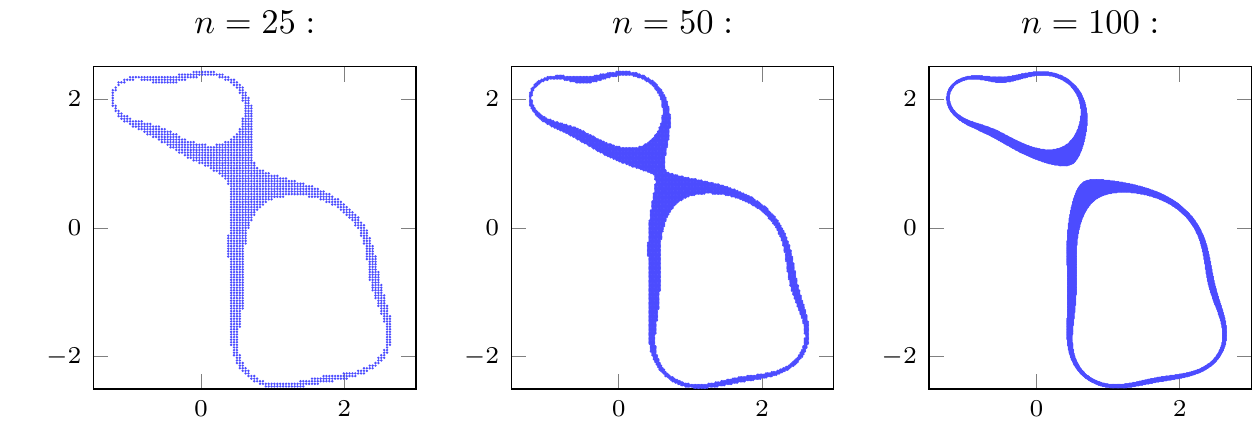}	
		\caption{Output of $\Gamma_n(A)$ in the complex plane for different values of $n$.}
		\label{fig:matrices}
	\end{figure}
\end{example}
The Matlab implementation that produced Figure \ref{fig:matrices} is available online at \url{https://github.com/frank-roesler/TriSpec}.

\section{Schr\"odinger Operators with Periodic Potentials}
 \label{sec:schrodinger}

In this section we prove Theorem \ref{th:main2} regarding the spectrum of Schr\"odinger operators $H=-\Delta+V$ with periodic potentials $V:\R^d\to\C$. Again, this is done by defining an explicit algorithm. We shall consider three computational problems which only differ in their primary set $\Om$, and are as follows (primary sets are defined immediately below):

\begin{align*}
\left\{
\begin{array}{rl}
	\hfill \Omega &=\quad \Omega^{\mathrm{Sch}}\quad\text{or}\quad \Omega^{\mathrm{Sch}}_p\quad\text{or}\quad \Omega^{\mathrm{Sch}}_{p,M}\\[1mm]
	\mathcal M &=\quad \big(\{K\subset\C\,|\, K \text{ closed}\}, \, d_{\mathrm{AW}}\big)\\[1mm]
	\Lambda &=\quad \{V\mapsto V(x) \,|\, x\in\R^d \}\\[2mm]
	\Xi &: \quad\Om \to \mathcal M;\quad
	V\mapsto \sigma(-\Delta+V),
\end{array}
\right.
\end{align*}
where $-\Delta+V$ is meant to be defined on $L^2(\R^d)$ with domain $H^2(\R^d)$, and $d_{\text{AW}}$ denotes the Attouch-Wets metric, which is a generalization of the Hausdorff metric for the case of sets which may be unbounded (see Appendix \ref{sec:metrics} for a brief discussion). Note that the spectrum $\sigma(-\Delta+V)$ is always non-empty in this case, so taking this metric makes sense.   For $p>d$ and $M>0$, the primary sets are defined as follows:
	\begin{align*}
	\Omega^{\mathrm{Sch}}_{p} &:= \{ V:\R^d\to\C \,|\, V \text{ is 1-periodic and }V|_{(0,1)^d}\in W^{1,p}((0,1)^d)  \},\\
	\Omega^{\mathrm{Sch}} &:=\bigcup_{p>d}\Omega^{\mathrm{Sch}}_p,\\
	\Om^{\mathrm{Sch}}_{p,M} &:= \{V\in\Om^{\mathrm{Sch}}_p\,|\, \|V\|_{W^{1,p}((0,1)^d)}\leq M\}.\\
	\end{align*}

Note that by Morrey's inequality, every $V\in\Omega^{\mathrm{Sch}}_p$ is continuous, and so the evaluation set $\Lambda$ which comprises point evaluations of $V$, is well-defined. In the language of the Solvability Complexity Index, the two parts of Theorem \ref{th:main2} can be expressed as follows:

\begin{itemize}
\item
Part (i) amounts to proving  that the computational problem for $\Omega^{\mathrm{Sch}}$ has an $\SCI$ value of $1$ (or, equivalently,   it belongs to $\Delta_2$). 
\item
Part (ii) amounts to showing that the computational problem for $\Om^{\mathrm{Sch}}_{p,M}$ belongs to $\Pi_1$, i.e. it can be approximated from above with explicit error bounds.
\end{itemize}

The proof of Theorem  \ref{th:main2}  is contained in  Subsection \ref{sec:proof-thm2}, and it relies on the following \emph{weaker} theorem:

\begin{theorem}\label{th:periodic_potential_mainth}

		The computational problem for $\Om^{\mathrm{Sch}}_{p}$ can be solved in one limit: $\SCI(\Om^{\mathrm{Sch}}_p) = 1$ (equivalently, $\Om^{\mathrm{Sch}}_p\in\Delta_2$).
\end{theorem}

 Note that this theorem is evidently weaker than Theorem \ref{th:main2}(i) as the class of potentials $\Omega^{\mathrm{Sch}}_p$ considered here is a strict subset of the class $\Omega^{\mathrm{Sch}}=\bigcup_{p>d}\Omega^{\mathrm{Sch}}_p$ considered in Theorem \ref{th:main2}(i). The proof is constructive, i.e. we  provide an explicit algorithm that computes the spectrum of any given operator with $V\in \Om^{\mathrm{Sch}}_p$.
 Note that this problem is fundamentally different from the discrete problem \eqref{eq:periodic_computational_problem} where we could directly access  the matrix elements of the (discrete) operator. Instead, the evaluation set $\Lambda$ gives access to the point values of the potential. Hence our task is to construct a sequence of algorithms $\{\Gamma_n\}_{n\in\N}$, such that each $\Gamma_n$ computes its output from finitely many point evaluations of $V$ using finitely many algebraic operations.\\
 
 The proof of Theorem \ref{th:periodic_potential_mainth} is contained in Subsection \ref{sec:proof-weak-thm}. Prior to that, Subsection \ref{sec:floquet} is dedicated to the Floquet-Bloch transform, in Subsection \ref{sec:potential-estimates} we approximate the potential $V$ and in Subsection \ref{sec:provisional_algorithm} we define a provisional algorithm.

\subsection{Floquet-Bloch transform}
\label{sec:floquet}
The Floquet-Bloch transform for Schr\"odinger operators with periodic potentials is well-studied. The following lemma is a collection of results in \cite{RS4} (below, $\mathcal S(\R^d)$ denotes the Schwartz space of rapidly decaying functions).
\begin{lemma}[{\cite[Ch. XIII.16]{RS4}}]\label{lemma:Floquet}
	For $f\in\mathcal S(\R^d)$ and $\theta\in[0,2\pi]^d$, define the map
	\begin{align*}
		(\cU_\theta f)(x) &:= (2\pi)^{-\f d2}\sum_{n\in\Z^d} f(x+n)e^{\i \theta\cdot(x+n)}.
	\end{align*}
	Then the following hold.
	\begin{enumi}
		\item $\cU_\theta$ extends uniquely to a bounded operator on $L^2(\R^d)$;
		\item For any $f\in L^2(\R^d)$, $\cU_\theta f$ is 1-periodic;
		\item The map
 		\begin{align*}
 			\cU : L^2(\R^d) &\to \int_{[0,2\pi]^d}^\oplus L^2((0,1)^d)\, d\theta \\
 			f &\mapsto (\cU_\theta f)_{\theta\in[0,2\pi]^d}
 		\end{align*}
 		is unitary;
 		\item The inverse $\cU^{-1}$ is given by 
 		\begin{align*}
 			(\cU^{-1}g)(x+n) = (2\pi)^{-\f d2} \int_{[0,2\pi]^d} g(x,\theta)e^{-\i n \theta} \,d\theta;
 		\end{align*}
 		\item For $H=-\Delta+V$, $V\in\Om_p$ one has
 		\begin{align*}
 			\cU H \cU^{-1} &= \int_{[0,2\pi]^d}^\oplus H(\theta)\,d\theta,
 		\end{align*}
 		where 
 		\begin{align}
 			H(\theta) &= -(\nabla+\i\theta)^2 + V \label{eq:H(theta)},
 			\\
 			\dom(H(\theta)) &= H^2_{\mathrm{per}}\big((0,1)^d\big) \qquad \forall \theta\in[0,2\pi]^d.\nonumber
 		\end{align}
 		Moreover, the map $\theta\mapsto H(\theta)$ is analytic and
 		\begin{align}\label{eq:spectral_identity}
 			\sigma(H) = \bigcup_{\theta\in[0,2\pi]^d} \sigma(H(\theta)).
 		\end{align}
	\end{enumi}
\end{lemma} 
The spectral identity \eqref{eq:spectral_identity} follows from the unitarity of $\mathcal U$ and a straightforward calculation, noting that $\bigcup_{\theta\in[0,2\pi]^d} \sigma(H(\theta))$ is closed by analyticity and periodicity in $\theta$.
\subsection{Estimating and Approximating the Potential}
\label{sec:potential-estimates}
The critical step is to compute approximations to the spectrum $\sigma(-\Delta+V)$   using only finitely many pointwise evaluations of $V$. This is done using the  Floquet-Bloch transform in conjunction with the Birman-Schwinger principle. The approximate potential is defined in \eqref{eq:approximated_fourier} and the critical error bound is stated in Lemma \ref{lemma:main_error_bound}.

\subsubsection{Birman-Schwinger principle}\label{sec:Birman-Schwinger}
Let $p>d,\; V\in\Omega_p^{\mathrm{Sch}}, \;\theta\in[0,2\pi]^d$ and let $H(\theta)$ be the corresponding Floquet-Bloch operator as in \eqref{eq:H(theta)}. Expanding the operator square, $H(\theta)$ can be written as
\begin{align*}
	H(\theta) &= -\Delta - 2\i\theta\cdot\nabla + |\theta|^2 + V.
\end{align*}
Let us choose the following decomposition of $H(\theta)$. We define
\begin{subequations} \label{eq:decomposition}
\begin{align}
	H_0 &:= -\Delta+1,  &\dom(H_0) &= H^2_{\text{per}}\big((0,1)^d\big),
	\\
	B(\theta) &:= - 2\i\theta\cdot\nabla + |\theta|^2 - 1 + V ,  &\dom(B(\theta)) &= H^1_{\text{per}}\big((0,1)^d\big).
\end{align} 
\end{subequations}
Then clearly one has $H(\theta) = H_0+B(\theta)$.
The auxiliary constant 1, which is added in $H_0$ and subtracted again in $B(\theta)$ was chosen for convenience, so that $H_0$ becomes a positive invertible  operator. Note that $B(\theta)$ is relatively compact with respect to $H_0$.
 For $\lambda\notin\sigma(H_0)$ one has
\begin{align*}
	\lambda - H_0 - B(\theta) &= H_0^{\f12}\big(I - H_0^{-\f12}B(\theta)(\lambda - H_0)^{-1}H_0^{\f12}\big)H_0^{-\f12}(\lambda - H_0),
\end{align*}
where $I$ denotes the identity operator on $L^2((0,1)^d)$ It follows that $\lambda - H_0 - B(\theta)$ is invertible if and only if $I - H_0^{-\nf12}B(\theta)(\lambda - H_0)^{-1}H_0^{\nf12}$ is invertible, and in that case
\begin{align*}
	(\lambda - H_0 - B(\theta))^{-1} &= (\lambda - H_0)^{-1}H_0^{\f12} \big(I - H_0^{-\f12}B(\theta)H_0^{\f12} (\lambda - H_0)^{-1}\big)^{-1}H_0^{-\f12}.
\end{align*}
 This identity (sometimes called the Birman-Schwinger principle) implies that 
 \begin{align*}
 	\lambda\in\C\setminus\sigma(H_0) \text{ is in }\sigma(H(\theta))
 	\quad \Leftrightarrow \quad
 	1\in\sigma\big(H_0^{-\f12}B(\theta)H_0^{\f12}(\lambda - H_0)^{-1}\big).
 \end{align*}
\subsubsection{Schatten class estimates} \label{sec:Schatten_estimates}
 We now study the analytic operator valued function 
 \begin{align*}
 	 K(\lambda,\theta) := H_0^{-\f12}B(\theta)H_0^{\f12}(\lambda - H_0)^{-1}.
 \end{align*}
We choose a Fourier basis for $L^2((0,1)^d)$, that is, we choose a numbering $\N\ni j\mapsto k_j\in 2\pi\Z^d$ such that $|k_j|$ is monotonically increasing  with $j$ and set 
\begin{equation*}
e_j := e^{\i k_j\cdot x}.
\end{equation*}
We note that $e_j\in\dom(H(\theta))$ for all $\theta$ and $\|e_j\|_{L^2((0,1)^d)}=1$ for all $j\in\N$. In this basis, the operators $H_0^{\nf12}$, $H_0$ and $\theta\cdot\nabla$ are all diagonal and one has
\begin{align}
	H_0^{\nf12} &= \diag\big((1+|k_j|^2)^{\nf12}\big) \label{eq:sqrt(H)_matrix}
	\\
	\lambda - H_0 &= \diag\big(\lambda - 1 -  |k_j|^2\big)\notag
	\\
	-2\i\theta\cdot\nabla &= \diag(2\,\theta\cdot k_j). \label{eq:theta*nabla_matrix}
\end{align} 
Therefore, we have 
\begin{align*}
	H_0^{-\f12}(-2\i\theta\cdot\nabla) H_0^{\f12}(\lambda - H_0)^{-1} &= \diag\left( \f{2\,\theta\cdot k_j}{\lambda - 1 - |k_j|^2} \right).
\end{align*}
Now the following lemma is easily proved.
\begin{lemma}[Schatten bound for $K$]\label{lemma:K_schatten_bound}
	For every $s>d$ the operator $K(\lambda,\theta)$ belongs to the Schatten class $\mathcal{C}_s$ and one has
	\begin{align*}
		\|K(\lambda,\theta)\|_{\mathcal{C}_s} &\leq \left(\f{2}{\pi}|\theta|  + \f{2}{\pi}\bigl\||\theta|^2 - 1 + V \bigr\| 
 		\right) C_\lambda \left(1-\f{d}{s}\right)^{-\f1s},
	\end{align*}
	where $C_\lambda := \sup_{j\in\N}\big| 1-\f{\lambda-1}{|k_j|^2} \big|^{-1}$ and $\|\cdot\|$ denotes the $L^2$ operator norm. 
\end{lemma} 
 \begin{proof}
 	Let $\lambda\in\C\setminus\sigma(H_0)$ and note that $C_\lambda<+\infty$ by our choice of $\lambda$. Observe that simple geometric considerations lead to the bound
		\begin{equation*}
		|k_j| \geq \pi j^{\f1d}.
		\end{equation*}
 	Then one has
 	\begin{align}
 		\left| \f{2\,\theta\cdot k_j}{\lambda - 1 - |k_j|^2}  \right| &\leq 2\f{|\theta| |k_j|}{|\lambda - 1 - |k_j|^2|} \nonumber
 		\\
 		&\leq 2C_\lambda |\theta| |k_j|^{-1} \nonumber
 		\\
 		&\leq \f2\pi C_\lambda |\theta| j^{-\f1d}.\label{eq:matrix_decay}
 	\end{align} 
 	Hence the characteristic numbers of $H_0^{-\f12}(-2\i\theta\cdot\nabla) H_0^{\f12}(\lambda - H_0)^{-1}$ are bounded by $\f2\pi C_\lambda |\theta| j^{-\f1d}$ and thus one has 
 	\begin{align*}
 		H_0^{-\f12}(-2\i\theta\cdot\nabla) H_0^{\f12}(\lambda - H_0)^{-1}
 		\in \mathcal{C}_s \quad\text{ for any }\quad s>d
 	\end{align*}
 	and
 	\begin{align}
 		\bigl\| H_0^{-\f12}(-2\i\theta\cdot\nabla) H_0^{\f12}(\lambda - H_0)^{-1} \bigr\|_{\mathcal{C}_s}
 		&\leq \f2\pi C_\lambda |\theta| \bigg(\sum_{j=1}^\infty j^{-\f sd}\bigg)^{\f1s}
 		\nonumber
 		\\
 		&\leq \f2\pi C_\lambda |\theta| \left(1-\f{d}{s}\right)^{-\f1s}.
 		\label{eq:HDHR_bound}
 	\end{align}
 	Next we turn to the potential term in $K(\lambda,\theta)$, that is, the operator $H_0^{-\nf12}\big(|\theta|^2 - 1 + V \big) H_0^{\nf12}(\lambda - H_0)^{-1}$. This is easily treated by the ideal property of $\mathcal{C}_s$, since the operator $|\theta|^2 - 1 + V$ is bounded. Indeed, we have for every $s>d$
 	\begin{align}
 		\bigl\| H_0^{-\nf12}\big(|\theta|^2 - 1 + V \big) H_0^{\nf12}(\lambda - H_0)^{-1} \bigr\|_{\mathcal{C}_s}
 		&\leq \bigl\| H_0^{-\nf12}\bigr\| \bigl\||\theta|^2 - 1 + V \bigr\| \bigl\|H_0^{\nf12}(\lambda - H_0)^{-1} \bigr\|_{\mathcal{C}_s}
 		\nonumber
 		\\
 		&\leq 
 		\bigl\||\theta|^2 - 1 + V \bigr\| 
 		\f{2}{\pi} C_\lambda \left(1-\f{d}{s}\right)^{-\f1s},
 		\label{eq:resolvent_decay}
 	\end{align}
	where the last line follows from a similar calculation to \eqref{eq:matrix_decay} and the fact that $\|H_0^{-\nf12}\|=1$  (this follows from the matrix representation \eqref{eq:sqrt(H)_matrix} and the fact that $k_1=0$).
 \end{proof}
 \begin{lemma}[Lipschitz continuity of $K$]\label{lemma:K_Lipschitz_bound}
	For every $s>d$ and $\lambda,\mu\in\C\setminus\sigma(H_0)$, $\theta,\vartheta\in[0,2\pi]^d$ one has
	\begin{align*}
		\|K(\lambda,\theta) - K(\mu,\vartheta)\|_{\mathcal{C}_s} &\leq 8 d^{\f12}  C_\lambda \left(1-\f{d}{s}\right)^{-\f1s} \big(|\theta-\vartheta| +  c_V C_\mu |\lambda-\mu| \big),
	\end{align*}
	where $C_\lambda$ is defined as in Lemma \ref{lemma:K_schatten_bound} and $c_V = \f12 + \pi d^{\f12} + (4\pi d^{\f12})^{-1} \|V-1\|_\infty$.
\end{lemma} 
\begin{proof}
	First, note that 
	\begin{align}\label{eq:B-B}
		B(\theta) - B(\vartheta) = -2\i(\theta-\vartheta)\cdot\nabla + (\theta+\vartheta)\cdot(\theta-\vartheta).
	\end{align}
	Using \eqref{eq:B-B} together with the resolvent identity for $(\lambda-H_0)^{-1}$ one obtains
	\begin{align*}
		K(\lambda,\theta) - K(\mu,\vartheta) 
		&= H_0^{-\f12}(-2\i(\theta-\vartheta)\cdot\nabla + (\theta+\vartheta)\cdot(\theta-\vartheta))H_0^{\f12}(\lambda-H_0)^{-1}
		\\
		&\hspace{4cm} + H_0^{-\f12} B(\vartheta)H_0^{\f12}(\lambda-\mu)(\lambda-H_0)^{-1}(\mu-H_0)^{-1}
		\\
		&= (\theta-\vartheta)\cdot(-2\i\nabla + \theta + \vartheta)(\lambda-H_0)^{-1} + (\lambda-\mu)K(\lambda,\vartheta)(\mu-H_0)^{-1}.
	\end{align*}
	Taking norms in $\mathcal C_s$ and using the estimates from the proof of Lemma \ref{lemma:K_schatten_bound} (and in particular  \eqref{eq:HDHR_bound}) we obtain the bound
	\begin{align*}
		\|K(\lambda,\theta) - K(\mu,\vartheta)\|_{\mathcal C_s}
		&\leq |\theta-\vartheta| \big\|(-2\i\nabla + \theta+\vartheta)(\lambda-H_0)^{-1} \big\|_{\mathcal C_s} + |\lambda-\mu| \big\|K(\lambda,\vartheta)\big\|_{\mathcal C_s} \big\|(\mu-H_0)^{-1}\big\|
		\\
		&\leq |\theta-\vartheta|\f2\pi C_\lambda |\theta+\vartheta| \left(1-\f{d}{s}\right)^{-\f1s}
		+|\lambda-\mu| \big\|K(\lambda,\vartheta)\big\|_{\mathcal C_s} C_\mu.
	\end{align*}
	Finally, applying Lemma \ref{lemma:K_schatten_bound} and using $|\theta+\vartheta|\leq 4\pi d^{\f12}$ we obtain
	\begin{align*}
		\|K(\lambda,\theta) - K(\mu,\vartheta)\|_{\mathcal C_s}
		&\leq  |\theta-\vartheta| 8 C_\lambda d^{\f12} \left(1-\f{d}{s}\right)^{-\f1s}
		+|\lambda-\mu|  
		\left(4d^{\f12} + 8\pi d + \f2\pi \|V-1\|
 		\right) C_\lambda C_\mu \left(1-\f{d}{s}\right)^{-\f1s}
 		\\
 		&\leq 8 d^{\f12}  C_\lambda \left(1-\f{d}{s}\right)^{-\f1s}
 		\left(
	 		|\theta-\vartheta| 
			+  
			\left(\f12 + \pi d^{\f12} + (4\pi d^{\f12})^{-1} \|V-1\|
	 		\right) C_\mu |\lambda-\mu|
 		\right).
	\end{align*}
	This concludes the proof.
\end{proof}

 While Lemma \ref{lemma:K_Lipschitz_bound} gives precise information about the dependence of the Lipschitz constant of $K$ on all parameters, it will be useful for us to have a bound which is less precise but more explicit (and manifestly computable).
 \begin{corollary}\label{cor:K_Lipschitz_crude}
	Let $s>d$ and $\delta, R>0$. Then for any $\theta,\vartheta\in[0,2\pi]^d$ and $\lambda,\mu\in\C$ such that $|\lambda-(1+|k_j|^2)|, |\mu-(1+|k_j|^2)|>\delta$ for all $j$ and $|\lambda|,|\mu|<R$ one has
	\begin{align*}
		\|K(\lambda,\theta) - K(\mu,\vartheta)\|_{\mathcal{C}_s} &\leq \delta^{-2}48R^2\f{sd}{s-d}\f{3p-d}{p-d}\big(1+\|V\|_{W^{1,p}((0,1)^d)}\big) \big(|\theta-\vartheta| + |\lambda-\mu| \big),
	\end{align*}
\end{corollary}
\begin{proof}
	This follows immediately from Lemma \ref{lemma:K_Lipschitz_bound} via rather crude estimates, noting that $C_\lambda\leq R\delta^{-1}$ and $\|V\|_\infty \leq \f{3p-d}{p-d}\|V\|_{W^{1,p}((0,1)^d)}$ (cf. \cite[Ch. 9.3]{Brezis}).
\end{proof}
\subsubsection{Approximation of the potential}
Next we study the matrix representation of the potential $V$ in the Fourier basis $\{e_j\}_{j\in\N}$ and its approximations. First, we observe that in the Fourier basis $\{e_j\}_{j\in\N}$ one has 
 	\begin{align*}
 		\p{e_j , Ve_m} = \hat V_{k_m-k_j}
 	\end{align*}
 	where $\hat V_k$ denote the Fourier coefficients of $V$. Indeed, a direct calculation gives
 	\begin{align*}
 		\p{e_j , Ve_m} = \int_{(0,1)^d} \overline{e^{\i k_j\cdot x}}V(x)e^{\i k_m\cdot x}\,dx
 		= \int_{(0,1)^d} V(x)e^{ \i(k_m-k_j)\cdot x} \, dx
 		= \hat V_{k_m-k_j}.
 	\end{align*}
 Now, we want to build a computable, finite size approximation of the matrix $(V_{jm}) = (\p{e_j , Ve_m})$. We start by approximating the Fourier coefficients $\hat V_k$.
 \begin{lemma}\label{lemma:integral_approx}
 	Let $n\in\N$ and define the lattice $I_n:=\{\f{m}{n}\,|\,m=0,\dots,n-1\}^d\subset (0,1)^d$.
 	For every $f\in W^{1,p}((0,1)^d)$, $p>d$, one has
 	\begin{align*}
 		\left| n^{-d}\sum_{\xi\in I_n} f(\xi) - \int_{(0,1)^d} f(x)\,dx \right| \leq \f{2n^{-1+\nf dp}}{1-\nf dp}\|\nabla f\|_{L^p((0,1)^d)}
 	\end{align*}
 \end{lemma}
 \begin{proof}
 	We write 
 	\begin{align*}
 		\int_{(0,1)^d} f(x)\,dx = \sum_{\xi\in I_n} \int_{(0,\f1n)^d+\xi} f(x)\,dx.
 	\end{align*}
 	Then, comparing the two sums term by term, we have
 	\begin{align*}
 		\left|\int_{(0,\f1n)^d+\xi} f(x)\,dx - n^{-d} f(\xi)\right|
 		&= \left|\int_{(0,\f1n)^d+\xi} f(x) - f(\xi) \,dx\right|
 		\\
 		&\leq \int_{(0,\f1n)^d+\xi} |f(x) - f(\xi)| \,dx
 		\\
 		&\leq \int_{(0,\f1n)^d+\xi} \f{2n^{-1+\nf dp}}{1-\nf dp}\|\nabla f\|_{L^p((0,1)^d)} \,dx
 		\\
 		&= \f{2n^{-1+\nf dp}}{1-\nf dp}\|\nabla f\|_{L^p((0,1)^d)} \int_{(0,\f1n)^d+\xi} \,dx
 		\\
 		&= \f{2n^{-1+\nf dp}}{1-\nf dp}\|\nabla f\|_{L^p((0,1)^d)} n^{-d},
 	\end{align*}
 	where Morrey's inequality was used in the third line (cf. (28) in the proof of \cite[Th. 9.12]{Brezis}). Summing these inequalities over $I_n$, we finally obtain
 	\begin{align*}
 		\left| n^{-d}\sum_{\xi\in I_n} f(\xi) - \int_{(0,1)^d} f(x)\,dx \right| 
 		&\leq \sum_{\xi\in I_n} \left|\int_{(0,\f1n)^d+\xi} f(x)\,dx - n^{-d} f(\xi)\right|
 		\\
 		&\leq \sum_{\xi\in I_n} \f{2n^{-1+\nf dp}}{1-\nf dp}\|\nabla f\|_{L^p((0,1)^d)} n^{-d}
 		\\
 		&\leq  \f{2n^{-1+\nf dp}}{1-\nf dp}\|\nabla f\|_{L^p((0,1)^d)}.
 	\end{align*}
 \end{proof}
 Let us introduce the \emph{approximate Fourier coefficients} for $k\in2\pi\Z^d$ and $n\in\N$,
 \begin{align}\label{eq:approximated_fourier}
 	\hat V_k^{\text{appr},n} := n^{-d}\sum_{\xi\in I_n} V(\xi)e^{\i k\cdot \xi}.
 \end{align}
 Note that the $\hat V_k^{\text{appr},n}$ can be computed in finitely many operations from the information provided in $\Lambda$.
  Lemma \ref{lemma:integral_approx} applied to the function $f(x)=V(x)e^{\i k\cdot x}$ leads to the error estimate
\begin{align}
	\left| \hat V_k^{\text{appr},n} - \hat V_k \right| 
	&\leq \f{2n^{-1+\nf dp}}{1-\nf dp}\|\nabla (Ve^{\i k\cdot x})\|_{L^p((0,1)^d)}
	\nonumber
	\\
	&\leq \f{2n^{-1+\nf dp}}{1-\nf dp}\big(\|\nabla V\,e^{\i k\cdot x}\|_{L^p((0,1)^d)} + \| V\nabla e^{\i k\cdot x}\|_{L^p((0,1)^d)} \big)
	\nonumber
	\\
	&= \f{2n^{-1+\nf dp}}{1-\nf dp}\big(\|\nabla V\|_{L^p((0,1)^d)} + |k|\,\| V\|_{L^p((0,1)^d)} \big)
	\nonumber
	\\
	&\leq \f{2n^{-1+\nf dp}}{1-\nf dp}(1+|k|)\| V\|_{W^{1,p}((0,1)^d)}
	\label{eq:Fourier_approximation_estimate}
\end{align}
 Next, we define the approximate potential matrix
 \begin{align*}
 	V^{\text{appr},n} := \big( \hat V^{\text{appr},n}_{k_m-k_j} \big)_{m,j\in\N}.
 \end{align*}
 We remark that the approximated potential matrix cannot be computed in finitely many algebraic operations from the point values of $V$, because it has infinitely many entries. This issue will be addressed next.
\begin{lemma}[Main Error Bound]\label{lemma:main_error_bound}
	For $N\in\N$ let $\h_N=\mathrm{Span}\{e_1,\dots,e_N\}$ and let $P_N:L^2((0,1)^d)\to\h_N$ be the orthogonal projection. Moreover, define
	 \begin{align*}
 	 K^{\mathrm{appr}}_{n}(\lambda,\theta) := H_0^{-\f12}\big(- 2\i\theta\cdot\nabla + |\theta|^2 - 1 + V^{\mathrm{appr},n}\big)H_0^{\f12}(\lambda - H_0)^{-1}.
	 \end{align*}
	Then for every $s>d$ one has
	\begin{align*}
		\big\|K(\lambda, \theta) - P_NK^{\mathrm{appr}}_{n}(\lambda, \theta)P_N \big\|_{\mathcal{C}_s} 
		&\leq
		C_\lambda\bigg(
		C_{s,d}^1 |\theta| N^{\f{1}{s}-\f{1}{d}}
		+
		C_{s,p,d}^2 \f{N^{1+\f{1}{d}}}{n^{1-\f dp}} \| V\|_{W^{1,p}((0,1)^d)}
		\\
		&\hspace{3cm}
		+
		C_{s,d}^3 N^{\f1s-\f1d} \Big(\big||\theta|^2-1 \big| + \|V\|_{L^\infty((0,1)^d)} \Big)
		\bigg),
	\end{align*}
	where we recall that $C_\lambda = \sup_{j\in\N}\big| 1-\f{\lambda-1}{|k_j|^2} \big|^{-1}$ and $C_{s,d}^1 ,\,C_{s,p,d}^2, \,C_{s,d}^3$ are explicit constants independent of $n,N,\lambda,\theta$.
\end{lemma} 
\begin{proof}
Again, we denote by $\|\cdot\|$ the $L^2(\R^d)$ operator norm in this proof.
We first treat the $\theta\cdot\nabla$ term. By equations \eqref{eq:sqrt(H)_matrix} - \eqref{eq:theta*nabla_matrix} we have
\begin{align*}
H_0^{-\f12}(- 2\i\theta\cdot\nabla) H_0^{\f12}(\lambda - H_0)^{-1} - P_N H_0^{-\f12}(- 2\i\theta\cdot\nabla) H_0^{\f12}(\lambda - H_0)^{-1} P_N = \diag\left( \f{2\,\theta\cdot k_j}{\lambda - 1 - |k_j|^2} ;\; j>N \right)
\end{align*}
	and by \eqref{eq:matrix_decay} we can estimate the above as
\begin{align}
	\left\|\diag\left( \f{2\,\theta\cdot k_j}{\lambda - 1 - |k_j|^2} ;\; j>N \right) \right\|_{\mathcal{C}_s}
	&\leq \f{2}{\pi} |\theta| C_\lambda \Bigg(\sum_{j=N+1}^\infty j^{-\f sd}\Bigg)^{\f1s}
	\nonumber
	\\
	&\leq \f{2}{\pi} |\theta| C_\lambda \left(\int_N^\infty t^{-\f sd}\,dt \right)^{\f1s}
	\nonumber
	\\
	&\leq \f{2}{\pi} |\theta| C_\lambda \left(\f{N^{-\f{s}{d}+1}}{\f{s}{d}-1}\right)^{\f1s}
	\nonumber
	\\
	&= \f{2 |\theta| C_\lambda }{\pi (\f{s}{d}-1)^\f{1}{s} }  N^{-\f{1}{d}+\f{1}{s}}.
	\label{eq:grad_term_estimate}
\end{align}
 To estimate the next term, we denote $W_n:=|\theta|^2 - 1 + V^{\mathrm{appr},n}$ and $W:=|\theta|^2 - 1 + V$ for brevity, and note that the diagonal operators $H_0^{-\f12}$ and $(\lambda - H_0)^{-1}$ commute with $P_N$. Thus we have $P_NH_0^{-\f12} W_n H_0^{\f12}(\lambda - H_0)^{-1}P_N = H_0^{-\f12}P_N W_n P_N H_0^{\f12}(\lambda - H_0)^{-1}$.
 Then we calculate
\begin{align}
	\bigl\| H_0^{-\f12}\big(W - P_N W_n P_N\big) H_0^{\f12}(\lambda - H_0)^{-1} \bigr\|_{\mathcal{C}_s}
	&\leq \bigl\| H_0^{-\f12}\big(W - P_N W P_N\big) H_0^{\f12}(\lambda - H_0)^{-1} \bigr\|_{\mathcal{C}_s}
	\nonumber
	\\
	& \quad + \bigl\| H_0^{-\f12} P_N\big(W - W_n\big)P_N H_0^{\f12}(\lambda - H_0)^{-1} \bigr\|_{\mathcal{C}_s}
	\nonumber
	\\
	&\leq \bigl\| H_0^{-\f12}\big(W - P_N W P_N\big) H_0^{\f12}(\lambda - H_0)^{-1} \bigr\|_{\mathcal{C}_s} 
	\nonumber
	\\
	& \quad + \bigl\| H_0^{-\f12} P_N\big(V - V^{\mathrm{appr},n}\big)P_N H_0^{\f12}(\lambda - H_0)^{-1} \bigr\|_{\mathcal{C}_s}.
	\label{eq:bounded_error_term}
\end{align}
Let us first consider the second term on the right-hand side of \eqref{eq:bounded_error_term}.
\begin{align*}
	\bigl\| H_0^{-\f12} \! P_N\big(V \! - \! V^{\mathrm{appr},n}\big)P_N H_0^{\f12}(\lambda - H_0)^{-1} \bigr\|_{\mathcal{C}_s} 
	&\!\leq \bigl\|H_0^{-\f12}\bigr\| \bigl\|P_N\big(V - V^{\mathrm{appr},n}\big)P_N\bigr\| \bigl\|H_0^{\f12}(\lambda - H_0)^{-1} \bigr\|_{\mathcal{C}_s} 
	\\
	 &\leq  \f2\pi C_\lambda \Big(1-\f{d}{s}\Big)^{-\f1s} \bigl\|P_N\big(V - V^{\mathrm{appr},n}\big)P_N\bigr\|
	\\
	&\leq  \f2\pi C_\lambda \Big(1-\f{d}{s}\Big)^{-\f1s} \! N \f{2n^{\f dp - 1}}{1-\f dp}\big(1+\sup_{j\leq N}|k_j|\big)\| V\|_{W^{1,p}((0,1)^d)}
	\\
	&\leq  \f2\pi C_\lambda \Big(1-\f{d}{s}\Big)^{-\f1s} \! N \f{2n^{\f dp - 1}}{1-\f dp}\big(1+\pi d^{\f12} N^{\f{1}{d}}\big)\| V\|_{W^{1,p}((0,1)^d)}
\end{align*}
		 where the third line follows from \eqref{eq:Fourier_approximation_estimate}, with a similar calculation as in \eqref{eq:resolvent_decay}. To simplify notation, we collect all constants independent of $\lambda,n,N$ into one and write 
\begin{align}\label{eq:V_term_bound}
	\bigl\| H_0^{-\f12} P_N\big(V \! - \! V^{\mathrm{appr},n}\big)P_N H_0^{\f12}(\lambda - H_0)^{-1} \bigr\|_{\mathcal{C}_s}
	\leq  C_{\lambda}C_{s,p,d}^2 \f{N^{1+\nf{1}{d}}}{n^{1-\nf dp}} \| V\|_{W^{1,p}((0,1)^d)}
\end{align}
Next we turn to the first term on the right-hand side of \eqref{eq:bounded_error_term}. We add and subtract $P_N W$ and use the triangle inequality to obtain
\begin{align*}
	\bigl\| H_0^{-\f12}\big(W - P_N W P_N\big) H_0^{\f12}(\lambda - H_0)^{-1} \bigr\|_{\mathcal{C}_s} 
	&\leq \bigl\| H_0^{-\f12}\big(W - P_N W\big) H_0^{\f12}(\lambda - H_0)^{-1} \bigr\|_{\mathcal{C}_s} 
	\\
	&\quad + \bigl\| H_0^{-\f12}\big(P_N W - P_N W P_N\big) H_0^{\f12}(\lambda - H_0)^{-1} \bigr\|_{\mathcal{C}_s} 
	\\
	&= \bigl\| H_0^{-\f12}(I-P_N)W H_0^{\f12}(\lambda - H_0)^{-1} \bigr\|_{\mathcal{C}_s} 
	\\
	&\quad + \bigl\| H_0^{-\f12} P_N W(I - P_N) H_0^{\f12}(\lambda - H_0)^{-1} \bigr\|_{\mathcal{C}_s} 
	\\
	&\leq \bigl\|H_0^{-\f12}(I-P_N)\bigr\|_{\mathcal{C}_s}\bigl\|H_0^{\f12}(\lambda - H_0)^{-1} \bigr\| \|W\| 
	\\
	&\quad + \bigl\| H_0^{-\f12}\bigr\|  \bigl\|(I- P_N) H_0^{\f12}(\lambda - H_0)^{-1} \bigr\|_{\mathcal{C}_s} \|W\|.
\end{align*}
Next we note that, by \eqref{eq:sqrt(H)_matrix} - \eqref{eq:theta*nabla_matrix}, $\|H_0^{-\f12}\|= 1$ and $\|H_0^{\f12}(\lambda - H_0)^{-1}\|\leq 2 C_\lambda$. Therefore
\begin{align}
	\bigl\| H_0^{-\f12}\big(W - P_N W P_N\big) H_0^{\f12}(\lambda - H_0)^{-1} \bigr\|_{\mathcal{C}_s} 
	&\leq 
	\Big( 2 C_\lambda \bigl\|H_0^{-\f12}(I-P_N)\bigr\|_{\mathcal{C}_s} 
	\nonumber
	 \\
	 &\hspace{2cm} +  \bigl\|(I - P_N) H_0^{\f12}(\lambda - H_0)^{-1} \bigr\|_{\mathcal{C}_s} \Big) \|W\|.
	 \label{eq:finite_section_error}
\end{align}
Finally, we employ \eqref{eq:sqrt(H)_matrix} - \eqref{eq:theta*nabla_matrix} again to estimate the finite section error in \eqref{eq:finite_section_error}. A straightforward calculation shows that
\begin{align*}
	\bigl\|H_0^{-\f12}(I-P_N)\bigr\|_{\mathcal{C}_s} &\leq \pi^{-1} \f{N^{\f1s-\f1d}}{(\f sd - 1)^{\nf1s}},
	\\
	\bigl\|(I - P_N) H_0^{\f12}(\lambda - H_0)^{-1} \bigr\|_{\mathcal{C}_s} &\leq \f2\pi C_\lambda \f{N^{\f1s-\f1d}}{(\f sd - 1)^{\nf1s}}.
\end{align*}
Using these bounds in \eqref{eq:finite_section_error}, we finally obtain the error estimate
\begin{align}\label{eq:W_term_bound}
	\bigl\| H_0^{-\f12}\big(W - P_N W P_N\big) H_0^{\f12}(\lambda - H_0)^{-1} \bigr\|_{\mathcal{C}_s} 
	&\leq
	\f{4 C_\lambda N^{\f1s-\f1d}}{\pi(\f sd - 1)^{\nf1s}} \|W\|
	\\
	&\leq
	\f{4 C_\lambda N^{\f1s-\f1d}}{\pi(\f sd - 1)^{\nf1s}} \Big(\big||\theta|^2-1 \big| + \|V\|_{L^\infty((0,1)^d)} \Big).\notag
\end{align}
Combining \eqref{eq:grad_term_estimate}, \eqref{eq:V_term_bound} and \eqref{eq:W_term_bound} yields the assertion with constants
\begin{align*}
	C_{s,d}^1 &= \f{2 }{\pi (\f{s}{d}-1)^\f{1}{s} },
	\\
	C_{s,p,d}^2 &= \f2\pi \left(1-\f{d}{s}\right)^{-\f1s}  \f{2}{1-\f dp}\Big(1+\pi d^{\nf12} \Big),
	\\
	C_{s,d}^3 &= \f{4}{\pi(\f sd - 1)^{\nf1s}}.
\end{align*}
\end{proof}
\subsection{A Provisional Algorithm}\label{sec:provisional_algorithm}
Armed with the error estimates from the last section, we are now able to define an algorithm based on the operator $K(\lambda, \theta)$. We first recall a result on Lipschitz continuity of perturbation determinants.

\begin{theorem}[{\cite[Thm. 6.5]{Simon}}]\label{thm:Simon}
	For $m\in\N$, denote by $\det_m$ the perturbation determinant on $\mathcal C_m$ (cf. \cite[Sec. XI.9]{DS2}). Then there exists a constant $c_m$ such that
	\begin{align*}
		|\det\nolimits_m(I-A) - \det\nolimits_m(I-B)| \leq e^{c_m(1+\|A\|_{\mathcal C_m}+\|B\|_{\mathcal C_m})} \|A-B\|_{\mathcal C_m}
	\end{align*}
	for all $A,B\in \mathcal C_m$.
	Moreover, one has $c_m\leq e(2+\log(m))$.
\end{theorem}
Next, we define 
\begin{de}[Provisional Algorithm]\label{def:provisional_algorithm}
	Let $p>d$, $V\in\Om_p^{\mathrm{Sch}}$ and $z_0\in\C$. For $N\in\N$ let $\Theta_N = \big(\theta_1^{(N)},\dots,\theta_N^{(N)}\big)$ be a linearly spaced   lattice in $[0,2\pi]^d$ and let $L_N:=\f1N(\Z+\i\Z) \cap Q_{z_0}$, where $Q_{z_0}:=\left\{z\in\C \,\middle|\, |\im(z-z_0)|,|\re(z-z_0)|\leq \f12 \right\}$. Then we let $n(N):=N^{\lceil\alpha\rceil}$, with $\alpha=\f{1+2d^{-1}-p^{-1}}{1-p^{-1}d}$,%
	\footnote{
		The exponent $\alpha$ is chosen such that $N^{1+\nf{1}{d}}/n^{1-\nf dp}\leq N^{\nf1p - \nf1d}$ (cf. Lemma \eqref{lemma:main_error_bound}).
	}
	and define
	\begin{align}\label{eq:periodic_potential_alg}
		\Gamma_N^{z_0}(V) := \bigcup_{i=1}^N \left\{ z\in L_N\,\middle|\, \big|\det\nolimits_{\lceil p\rceil}\big(I-P_NK^{\mathrm{appr}}_{n(N)}\big(z,\theta_i^{(N)}\big)P_N\big)\big|\leq N^{-(\f{1}{2d} - \f{1}{2p})} \right\}.
	\end{align}
\end{de}
Note that for every $N\in\N$, $\Gamma_N^{z_0}(V)$ can be computed from the information in $\Lambda$ using finitely many algebraic operations (recall in particular the approximated Fourier coefficients \eqref{eq:approximated_fourier}). Therefore, every $\Gamma_N^{z_0}$ defines an arithmetic algorithm in the sense of Definition \ref{def:Algorithm}.
\begin{prop}[Convergence of Provisional Algorithm]\label{prop:conv_1}
	Let $z_0\in\C$ and let $V\in\Om_{p}^{\mathrm{Sch}}$. The following statements hold.
	\begin{enumi}
		\item For any sequence $z_N\in\Gamma^{z_0}_N(V)$ with $z_N\to z\in Q_{z_0}\setminus\bigcup_j(1+|k_j|^2)$ one has $z\in\sigma(-\Delta+V)$.
		\item For any $z\in\left(\sigma(-\Delta+V)\cap Q_{z_0}\right)\setminus\bigcup_j(1+|k_j|^2)$ there exists a sequence $z_N\in\Gamma^{z_0}_N(V)$ with $z_N\to z$ as $N\to+\infty$.
		\item Let $V\in\Om_{p,M}^{\mathrm{Sch}}$.
			For any given $\eps,\delta > 0$ one has $\left(\sigma(-\Delta+V)\cap Q_{z_0}\right)\setminus B_\delta\big(\bigcup_j(1+|k_j|^2)\big) \subset B_\eps(\Gamma_N^{z_0}(V))$ as soon as $N>\max \{ \eps^{-1},\, N_{\delta,z_0} \}$, where 
			\begin{equation}\label{eq:N_{delta,z_0}_definition}
				N_{\delta,z_0} = \Big[ C_{\delta,z_0}^\textnormal{Lip} + G(|z_0|+1)\delta^{-1} \Big]^{\f{2}{\nf{1}{p}+\nf{1}{d}}}
			\end{equation}
			and $C_{\delta,z_0}^\textnormal{Lip}$, $G$ are explicit constants, which can be computed in finitely many operations from $z_0,\delta,s,p,d$ and the \emph{a priori} bound $M$ for $\|V\|_{W^{1,p}}$.
	\end{enumi}
\end{prop}
\begin{proof}
(i) Let $z_N\in\Gamma^{z_0}_N(V)$ and assume that $z_N\to z$ for some $z\in\C$. We need to show that $z\in\sigma(-\Delta+V)$.
	Since $z_N\in\Gamma_N^{z_0}(V)$, we have $\det\nolimits_{\lceil p\rceil}\big(I-P_NK^{\mathrm{appr}}_{n(N)}(z,\theta_{i_N})P_N\big)\to 0$ for some sequence $\{\theta_{i_N}\}_{N\in\N}\subset[0,2\pi]^d$. Then there exists a convergent subsequence (again denoted by $\theta_{i_{N}}$) converging to some some $\theta\in[0,2\pi]^d$. We first note that due to Theorem \ref{thm:Simon}  we have the $N$-independent determinant error bound
	\begin{align*}
		\Big|\det\nolimits_{\lceil p\rceil}\big(I-K & (z_{N},\theta_{i_{N}})\big)
		 -
		 \det\nolimits_{\lceil p\rceil}  \big(I-P_{N}K^{\mathrm{appr}}_{n(N)}(z_{N},\theta_{i_{N}})P_{N}\big)\Big|
		 \leq 
		 \\
		 &\big\|K(z_{N}, \theta_{i_{N}}) - P_{N}K^{\mathrm{appr}}_{n(N)}(z_{N}, \theta_{i_{N}})P_{N} \big\|_{\mathcal{C}_{\lceil p\rceil}}
		 e^{
		 	c_{\lceil p\rceil}(
			 	1
			 	+\|K(z_{N}, \theta_{i_{N}})\|_{\mathcal{C}_{\lceil p\rceil}} 
			 	+ \|K^{\mathrm{appr}}_{n(N)}(z_{N}, \theta_{i_{N}})\|_{\mathcal{C}_{\lceil p\rceil}}
		 	)
		 	},
	\end{align*}
 (note that $\det_{\lceil p\rceil}(I-K(z_{N},\theta_{i_{N}}))$ is well-defined by Lemma \ref{lemma:K_schatten_bound} because $p>d$). We note that the exponential factor in the last line is uniformly bounded in $N$ by some explicit constant $c_{\textnormal{exp}}$ (cf. Lemmas \ref{lemma:K_schatten_bound} and \ref{lemma:main_error_bound}). 
 Using Lemma \ref{lemma:main_error_bound} with our choice $n(N)=N^{\lceil\alpha\rceil}$ and $s=p$, we obtain
	\begin{align}
		\Big| \det\nolimits_{\lceil p\rceil} \big(I & - K(z_{N},\theta_{i_{N}})\big)
		 -
		  \det\nolimits_{\lceil p\rceil} \big(I-P_{N} K^{\mathrm{appr}}_{n(N)}(z_{N},\theta_{i_{N}})P_{N}\big)\Big|
		 \leq 
		 \nonumber
		 \\
		 &
		 c_{\textnormal{exp}}\,C_{z_N}\bigg(
		C_{p,d}^1 |\theta_{i_N}| N^{\f{1}{p}-\f{1}{d}}
		+
		C_{p,p,d}^2 N^{\f1p-\f1d} \| V\|_{W^{1,p}}
		+
		C_{p,d}^3 N^{\f1p-\f1d} \Big(\big||\theta_{i_N}|^2-1 \big| + \|V\|_{L^\infty} \Big)
	\bigg),
	\label{eq:used_main_estimate}
	\end{align}
	Note that each term on the right hand side has a negative power of $N$ and therefore tends to 0 as $N\to+\infty$. Note further that $C_{z_N}=\sup_{j\in\N}\big| 1-\f{z_N-1}{|k_j|^2} \big|^{-1}$ remains bounded as $N\to+\infty$ by our assumption that $z\notin\bigcup_j(1+|k_j|^2)$. Therefore we have
	\begin{align*}
		\Big| \det\nolimits_{\lceil p\rceil} & \big(I-K(z_{N},\theta_{i_{N}})\big)
		 -
		  \det\nolimits_{\lceil p\rceil} \big(I-P_{N} K^{\mathrm{appr}}_{n(N)}(z_{N},\theta_{i_{N}})P_{N}\big)\Big|
		\to 0 \qquad\text{ as }N\to+\infty.
	\end{align*}
	To conclude the proof of (i), if we note that continuity in $(z,\theta)$ and periodicity in $\theta$ imply 
	\begin{align*}
		\left|\det\nolimits_{\lceil p\rceil}\big(I-K(z,\theta)\big)\right| 
		&= \lim_{N\to+\infty} \left|\det\nolimits_{\lceil p\rceil}\big(I-K(z_{N},\theta_{i_{N}})\big)\right|
		\\
		&\leq \lim_{N\to+\infty}
		 \left|\det\nolimits_{\lceil p\rceil}\big(I-P_{N}K^{\mathrm{appr}}_{n(N)}(z_{N},\theta_{i_{N}})P_{N}\big)\right|
		 \\
		 &\qquad +\lim_{N\to+\infty} \left|\det\nolimits_{\lceil p\rceil}\big(I-K(z_{N},\theta_{i_{N}})\big)
		 -
		 \det\nolimits_{\lceil p\rceil}\big(I-P_{N}K^{\mathrm{appr}}_{n(N)}(z_{N},\theta_{i_{N}})P_{N}\big)\right|
		\\
		 &= 0.
	\end{align*}
	Hence we have $1\in\sigma(K(z,\theta))$ and thus $z\in\sigma(-\Delta+V)$.

(ii) Conversely, denote $H:=-\Delta+V$ and let $z\in\sigma(H)\cap Q_{z_0}\setminus\bigcup_j(1+|k_j|^2)$. We need to  show that there exists a sequence $z_N\in\Gamma_N^{z_0}(V)$ such that $z_N\to z$. In fact, let $z_N\in L_N$ be any sequence with $|z-z_N|<\f1N$ (such a sequence exists by the definition of $L_N$). Since $z\in\sigma(H)$ there exists $\theta\in[0,2\pi]^d$ with $z\in\sigma(H(\theta))$, cf. Lemma \ref{lemma:Floquet}. Consequently
\begin{align*}
	\det\nolimits_{\lceil p\rceil}\big(I-K(z,\theta)\big) = 0.
\end{align*}
Consider some sequence $\theta_N\in\Theta_N$ with $|\theta_N-\theta|<\f1N$, the existence of which is guaranteed by the definition of $\Theta_N$. 
By Lemma \ref{lemma:K_Lipschitz_bound} and Theorem \ref{thm:Simon} there exists $C>0$ with
\begin{equation}
	\left|\det\nolimits_{\lceil p\rceil}\big(I-K(z_N,\theta_N)\big)\right| \leq C(|z-z_N| + |\theta-\theta_N|)
	\leq \f{2C}{N}.
	\label{eq:analytic_local_bound}
\end{equation}
Thus, using \eqref{eq:analytic_local_bound} and the main error estimate \eqref{eq:used_main_estimate}, we obtain
\begin{align*}
	\Big|\det\nolimits_{\lceil p\rceil} \! \big(I-P_{N} K^{\mathrm{appr}}_{n(N)}(z_{N},\theta_{N})P_{N}\big)\Big|
	 &\leq \left|\det\nolimits_{\lceil p\rceil} \! \big(I-K(z_N,\theta_N)\big)\right| 
	\\
	 &\quad + \left| \det\nolimits_{\lceil p\rceil} \! \big(I-P_{N} K^{\mathrm{appr}}_{n(N)}(z_{N},\theta_{N})P_{N}\big) - \det\nolimits_{\lceil p\rceil}\big(I-K(z_N,\theta_N)\big)\right|
	\\
	 &\leq \f{2C}{N} 
	+ c_{\textnormal{exp}}\,C_{z_N}N^{\f{1}{p}-\f{1}{d}}\bigg(
		C_{p,d}^1 |\theta_{N}| 
		+
		C_{p,d}^2  \| V\|_{W^{1,p}((0,1)^d)}
	\\
		&\hspace{4cm} + C_{p,d}^3 \Big(\big||\theta_{N}|^2-1 \big| + \|V\|_{L^\infty((0,1)^d)} \Big)
	\bigg).
\end{align*}
 To keep the notation simple, we collect all constants independent of $N$ into a single constant $G=G(p,d,V)$ (note that $|\theta_N|\leq 2\pi\sqrt d$ for all $N$). This gives
\begin{align}
	\Big|\det\nolimits_{\lceil p\rceil} \big(I-P_{N} K^{\mathrm{appr}}_{n(N)}(z_{N},\theta_{N})P_{N}\big)\Big|
	 &\leq  \f{2C}{N} 
	+ c_{\textnormal{exp}}\,G\,C_{z_N} N^{\f{1}{p}-\f{1}{d}} 
	\nonumber
	\\
	&\leq  N^{-(\f{1}{2d} - \f{1}{2p})},
	\label{eq:final_scaling}
\end{align}
 where the last line holds for $N$ large enough (note that $\f{1}{2d} - \f{1}{2p}>0$ by our assumption $p>d$ and that $C_{z_N}$ remains bounded as $N\to+\infty$ by our assumptions on $z$). Comparing \eqref{eq:final_scaling} to \eqref{eq:periodic_potential_alg}, we see that $z_N\in\Gamma_N^{z_0}(V)$ for $N$ large enough.

 (iii) The proof is similar to that of (ii). Let $z\in \left(\sigma(-\Delta+V)\cap Q_{z_0}\right)\setminus B_\delta\big(\bigcup_j(1+|k_j|^2)\big)$ be an arbitrary point and let $z_N\in L_N$ be any sequence with $|z-z_N|<\f1N$. Since $z\in\sigma(H)$ there exists $\theta\in[0,2\pi]^d$ with $z\in\sigma(H(\theta))$, cf. Lemma \ref{lemma:Floquet}. Consequently
\begin{align*}
	\det\nolimits_{\lceil p\rceil}\big(I-K(z,\theta)\big) = 0.
\end{align*}
Consider some sequence $\theta_N\in\Theta_N$ with $|\theta_N-\theta|<\f1N$, the existence of which is guaranteed by the definition of $\Theta_N$. 
By Corollary \ref{cor:K_Lipschitz_crude} and Theorem \ref{thm:Simon} this implies
\begin{align*}
	\left|\det\nolimits_{\lceil p\rceil}\big(I-K(z_N,\theta_N)\big)\right| 
	&\leq C_{\delta,z_0}^\textnormal{Lip}\big(1+\|V\|_{W^{1,p}((0,1)^d)}\big) \big(|\theta-\theta_N| + |z-z_N| \big)
	\\
	&\leq 2N^{-1}C_{\delta,z_0}^\textnormal{Lip}\big(1+\|V\|_{W^{1,p}((0,1)^d)}\big) 	\label{eq:quantitative_analytic_local_bound}
\end{align*}
where $C_{\delta,z_0}^\textnormal{Lip} = c_{\textnormal{exp}}\delta^{-2}48(|z_0|+1)^2\f{sd}{s-d}\f{3p-d}{p-d}$. Turning to the finite approximation $K^{\mathrm{appr}}_n$, this gives
\begin{align*}
	\left|\det\nolimits_{\lceil p\rceil}\!\big(I-P_{N} K^{\mathrm{appr}}_{n(N)}(z_{N},\theta_{N})P_{N}(z_N,\theta_N)\big)\right| 
	&\leq \left|\det\nolimits_{\lceil p\rceil}\big(I-K(z_N,\theta_N)\big)\right|
	\\
	& + \left|\det\nolimits_{\lceil p\rceil}\big(I-K(z_N,\theta_N)\big) - \det\nolimits_{\lceil p\rceil}\big(I-K^{\mathrm{appr}}(z_N,\theta_N)\big)\right|
	\\
	&\leq 2N^{-1}C_{\delta,z_0}^\textnormal{Lip}\big(1+\|V\|_{W^{1,p}((0,1)^d)}\big) + G C_{z_N} N^{\f{1}{p}-\f{1}{d}}
	\\
	&\leq 2N^{-1}C_{\delta,z_0}^\textnormal{Lip}\big(1+\|V\|_{W^{1,p}((0,1)^d)}\big) + G (|z_0|+1)\delta^{-1} N^{\f{1}{p}-\f{1}{d}},
\end{align*}
where $G$ denotes the explicit and computable constant introduced in step (ii) above and we have used the bound $C_{z_N}\leq (|z_0|+1)\delta^{-1}$.
The condition $z_N\in\Gamma_N(V)$ is thus implied by
\begin{align*}
	2N^{-1}C_{\delta,z_0}^\textnormal{Lip}(1+\|V\|_{W^{1,p}((0,1)^d)}) + G (|z_0|+1)\delta^{-1} N^{\f{1}{p}-\f{1}{d}}  < N^{-(\nf1{2d}-\nf1{2p})},
\end{align*} 
which immediately implies the assertion, noting that $|z-z_N|<\f1N<\eps$ as soon as $N>\eps^{-1}$.
\end{proof}
\subsection{Proof of Theorem \ref{th:periodic_potential_mainth}}
\label{sec:proof-weak-thm}
The provisional algorithm  $\Gamma_N^{z_0}$ defined in \eqref{eq:periodic_potential_alg} is not sufficient to prove Theorem \ref{th:periodic_potential_mainth}, because it does not approximate any eigenvalues lying in the set $\bigcup_j(1+|k_j|^2)$. This is ultimately due to the decomposition \eqref{eq:decomposition} into $H_0$ and $B(\theta)$. In order to solve this issue, we define a second provisional algorithm based on a different decomposition. We define
\begin{subequations} \label{eq:decomposition2}
\begin{align}
	\tilde H_0 &:= -\Delta+2,  &\dom(\tilde H_0) &= H^2_{\text{per}}\big((0,1)^d\big)
	\\
	\tilde B(\theta) &:= - 2\i\theta\cdot\nabla + |\theta|^2 - 2 + V ,  &\dom(\tilde B(\theta)) &= H^1_{\text{per}}\big((0,1)^d\big).
\end{align} 
\end{subequations}
Then, obviously, $\tilde H_0 + \tilde B(\theta)=H(\theta)$ and by the Birman-Schwinger principle,
 \begin{align*}
 	\lambda\in\C\setminus\sigma(\tilde H_0) \text{ is in }\sigma(H(\theta))
 	\quad \Leftrightarrow \quad
 	1\in\sigma\big(\tilde H_0^{-\f12}\tilde B(\theta)\tilde H_0^{\f12}(\lambda - \tilde H_0)^{-1}\big).
 \end{align*}
Moreover, we have $\sigma(\tilde H_0) = \bigcup_j(2+|k_j|^2)$ and 
\begin{lemma}\label{lemma:sigma(H)_sigma(Htilde)}
	One has
	\begin{align*}
		\sigma(H_0)\cap \sigma(\tilde H_0) = \emptyset.
	\end{align*}
\end{lemma}
\begin{proof}
	The assertion is equivalent to the following equation having no solutions for any integers $n_i,\tilde n_i$:
	\begin{align*}
		2 + (2\pi)^2\sum_{i=1}^d\tilde n_i^2 &= 1 + (2\pi)^2\sum_{i=1}^d n_i^2,
	\end{align*}
which becomes
	\begin{align*}
1 &= (2\pi)^2\sum_{i=1}^d (n_i^2 - \tilde n_i^2).
	\end{align*}
	Clearly the right-hand side of the above equation is always irrational or zero, while the left hand side is always nonzero rational. Thus no solution exists.
\end{proof}
Subsections \ref{sec:potential-estimates} and \ref{sec:provisional_algorithm} carry over trivially to the decomposition $\tilde H_0,\,\tilde B(\theta)$. Analogously to \eqref{eq:periodic_potential_alg} we define the new algorithm
\begin{de}[Second Provisional Algorithm]
	Let $p>d$, $V\in\Om_p^{\mathrm{Sch}}$ and $z_0\in\C$. For $N\in\N$ let $\Theta_N = \big(\theta_1^{(N)},\dots,\theta_N^{(N)}\big)$ be a linearly spaced lattice in $[0,2\pi]^d$ and let $L_N:=\f1N(\Z+\i\Z) \cap Q_{z_0}$, where $Q_{z_0}:=\left\{z\in\C \,\middle|\, |\im(z-z_0)|,|\re(z-z_0)|\leq \f12 \right\}$. Then we choose $n(N):=N^{\lceil\alpha\rceil}$ (with $\alpha$ as in Definition \ref{def:provisional_algorithm}) and define
	\begin{align}\label{eq:periodic_potential_alg2}
		\tilde \Gamma_N^{z_0}(V) := \bigcup_{i=1}^N \left\{ z\in L_N\,\middle|\, \big|\det\nolimits_{\lceil p\rceil}\big(I-P_N\tilde K^{\mathrm{appr}}_{n(N)}\big(z,\theta_i^{(N)}\big)P_N\big)\big|\leq N^{-(\f{1}{2d} - \f{1}{2p})} \right\},
	\end{align}
	where now 
	\begin{align*}
	 	 \tilde K^{\mathrm{appr}}_{n}(\lambda,\theta) := \tilde H_0^{-\f12}\big(- 2\i\theta\cdot\nabla + |\theta|^2 - 2 + V^{\mathrm{appr},n}\big)\tilde H_0^{\f12}(\lambda - \tilde H_0)^{-1}.
	\end{align*}
\end{de}
From the proofs in Subsections \ref{sec:potential-estimates} and \ref{sec:provisional_algorithm} we immediately obtain the following convergence result.
\begin{prop}[Convergence of Second Provisional Algorithm]\label{prop:conv_2}
	Let $z_0\in\C$ and let $V\in\Om_{p}^{\mathrm{Sch}}$. The following statements hold.
	\begin{enumi}
		\item For any sequence $z_N\in\tilde \Gamma^{z_0}_N(V)$ with $z_N\to z\in Q_{z_0}\setminus\bigcup_j(2+|k_j|^2)$ one has $z\in\sigma(-\Delta+V)$.
		\item For any $z\in\left(\sigma(-\Delta+V)\cap Q_{z_0}\right)\setminus\bigcup_j(2+|k_j|^2)$ there exists a sequence $z_N\in\tilde \Gamma^{z_0}_N(V)$ with $z_N\to z$ as $N\to+\infty$.
		\item Let $V\in\Om_{p,M}^{\mathrm{Sch}}$. For any given $\eps,\delta>0$ one has $\left(\sigma(-\Delta+V)\cap Q_{z_0}\right)\setminus B_\delta\big(\bigcup_j(2+|k_j|^2)\big) \subset  B_\eps(\tilde\Gamma_N^{z_0}(V))$ as soon as $N>\max \{ \eps^{-1},\, N_{\delta,z_0} \}$, where $N_{\delta,z_0}$ was defined in \eqref{eq:N_{delta,z_0}_definition}.
	\end{enumi}
\end{prop}
Armed with the provisional algorithms \eqref{eq:periodic_potential_alg} and \eqref{eq:periodic_potential_alg2}, we are finally able to define the main algorithm:
\begin{de}[Main Algorithm]\label{def:final_alg}
	Let $p>d$, $V\in\Om_p^{\mathrm{Sch}}$, and choose a numbering $\{Z_j\}_{j\in\N}$ of $\Z+\i\Z$ such that $|Z_i|\leq |Z_j|$ for $i\leq j$. For $N\in\N$ define
	\begin{align*}
		\Gamma_N(V) := \bigcup_{i=1}^N \Big(\Gamma_N^{Z_i}(V) \cup  \tilde \Gamma_N^{Z_i}(V) \Big).
	\end{align*}
\end{de}
Combining Propositions \ref{prop:conv_1} and \ref{prop:conv_2}, we can complete the proof of Theorem \ref{th:periodic_potential_mainth}:
\begin{proof}[Proof of Theorem \ref{th:periodic_potential_mainth}]
 Let $p>d$. For any $V\in\Om_p^{\mathrm{Sch}}$ we need to show that
	\begin{align*}
		d_{\mathrm{AW}}\big(\Gamma_N(V),\,\sigma(-\Delta+V)\big) \to 0.
	\end{align*}
	Indeed, by Propositions \ref{prop:conv_1} and \ref{prop:conv_2}, the following holds
	\begin{enuma}
		\item For any sequence $z_N\in\Gamma_N(V)$ with $z_N\to z\in \C$ one has $z\in\sigma(-\Delta+V)$.
		\item For any $z\in\sigma(-\Delta+V)$ there exists a sequence $z_N\in\Gamma_N(V)$ with $z_N\to z$ as $N\to+\infty$.
	\end{enuma}
	Attouch-Wets convergence now follows from a standard argument, cf. \cite[Prop. 2.8]{R19}.
\end{proof}

\subsection{Proof of Theorem \ref{th:main2}}
\label{sec:proof-thm2}
We can finally prove Theorem \ref{th:main2}.

\begin{proof}[Proof of part (i)]
The \emph{a priori} knowledge of $p$ can be removed by slightly modifying the provisional algorithms \eqref{eq:periodic_potential_alg} and \eqref{eq:periodic_potential_alg2}.	An algorithm that achieves $\SCI=1$ is obtained by modifying the determinant threshold in \eqref{eq:periodic_potential_alg} and \eqref{eq:periodic_potential_alg2}. If the cutoff $N^{-(\f{1}{2d} - \f{1}{2p})}$ is replaced by $\log(N)^{-1}$ and the discretization width $n(N)=N^{\lceil\alpha\rceil}$ is replaced by $n(N):=e^N$, the proof of Propositions \ref{prop:conv_1} and \ref{prop:conv_2} is valid independently of the value of $p$ (cf. eq. \eqref{eq:final_scaling}). 
\\
\\
	\emph{Proof of part (ii)}. Fix $R>0$. According to our choice of numbering $\{Z_i\}_{i\in\N}$ we have $|Z_i|>R$ for all $i>4R^2$ and thus $B_R(0)\subset \bigcup_{i\leq 4R^2} Q_{Z_i}$. Let $\delta_R = \frac12 \min\bigl\{|x-y|\,\big|\,x\in\sigma(H_0),\,y\in\sigma(\tilde H_0),\,|x|,|y|\leq R\bigr\}$, which is greater than $0$   by Lemma \ref{lemma:sigma(H)_sigma(Htilde)}. This choice implies that removing $\delta_R$-neighborhoods of the free spectra does not actually restrict our domain of computation:
	\begin{align*}
		(B_R(0)\setminus B_{\delta_R}(\sigma(H_0)))\cup (B_R(0)\setminus B_{\delta_R}(\sigma(\tilde H_0))) = B_R(0)
	\end{align*}
	for all $R>0$.
	Now let $\eps>0$. By Propositions \ref{prop:conv_1}(iii) and \ref{prop:conv_2}(iii) we have 
	\begin{equation}\label{eq:inclusion_without_AW}
		\sigma(-\Delta+V)\cap B_R(0) \subset B_\eps(\Gamma_N(V))
	\end{equation}
	as soon as $N > \max\{ 4R^2,\, \eps^{-1},\, N_{R}\}$, where $N_R = \max\{N_{\delta_R,Z_i}\,|\, i\leq 4R^2 \}$ (recall $N_{\delta,z}$ from \eqref{eq:N_{delta,z_0}_definition}).
	We recall that 
	\begin{itemize}
		\item the lower bound $4R^2$ ensures that $B_R(0)$ is covered by the search regions $Q_{Z_i}$;
		\item the lower bound $\eps^{-1}$ ensures that for every $z\in\sigma(-\Delta+V)$ there exists $z_N\in L_N$ with $|z-z_N|<\eps$;
		\item the lower bound $N_{R}$ ensures that $z_N\in\Gamma_N(V)$.
	\end{itemize}
	To conclude the proof of part (ii) (i.e., that $\Om_{p,M}^{\mathrm{Sch}}\in\Pi_1$), we need   a set $X_k$ as in Definition \ref{def:pi-sigma}. To this end, for $V\in\Om_{p,M}^{\mathrm{Sch}}$ and $k\in\N$, choose $N(k):=\max\{4k^2, 2^k, N_k\}$ and define 
	\[
	X_k := B_{2^{-k}}(\Gamma_{N(k)}(V))\cup (\C\setminus B_k(0)).
	\]
	Then by the definition of the Attouch-Wets distance $d_{\textnormal{AW}}$ (cf. Definition \ref{def:Attouch-Wets}) one has
	\begin{align*}
		d_{\mathrm{AW}}(\Gamma_{N(k)}(V),X_k) &= \sum_{n=1}^\infty 2^{-n}\min\Bigg\{ 1\,,\,\sup_{\substack{p\in\C\\|p|<n}}\left| \inf_{a\in \Gamma_{N(k)}(V)}|a-p| - \inf_{b\in X_k}|b-p| \right| \Bigg\}
		\\
		&\leq \sum_{n=1}^{k} 2^{-n}\min\Bigg\{ 1\,,\,\sup_{\substack{p\in\C\\|p|<k}}\left| \inf_{a\in \Gamma_{N(k)}(V)}|a-p| - \inf_{b\in X_k}|b-p| \right| \Bigg\} + \sum_{n=k+1}^\infty 2^{-n}
		\\
		&\leq 2^{-k} \sum_{n=1}^k 2^{-n} + 2^{-k}
		\\
		&\leq 2^{-k+1}
	\end{align*}
	where the third line follows from the definition of $X_k$. Moreover, by \eqref{eq:inclusion_without_AW} (with $\eps=2^{-k}$ and $R=k$) we have
	\begin{align*}
		\sigma(-\Delta+V) \subset X_k
	\end{align*}
	for all $k\in\N$. These facts, together with step (i) imply $\Om_{p,M}^{\mathrm{Sch}}\in\Pi_1$.
\end{proof}
\begin{remark}
	In view of Theorems \ref{th:matrix_mainth} and \ref{th:periodic_potential_mainth} it is natural to ask whether one might have $\Om_{p,M}^{\mathrm{Sch}}\in\Sigma_1$ (and therefore $\Om_{p,M}^{\mathrm{Sch}}\in\Delta_1$). This is indeed a nontrivial open problem. Recalling the proof of Theorem \ref{th:matrix_mainth} (in particular \eqref{eq:dist_bound}), the error bound for the approximation ``from below'' used the fact that $\det(zI-A(\theta))$ is a polynomial. In the proof of Theorem \ref{th:periodic_potential_mainth} on the other hand, the function $\det_{\lceil p\rceil}(I-K(z,\theta))$, whose zeros must be approximated, is only known to be \emph{analytic}. Obtaining $\Sigma_1$ classification amounts to obtaining explicit upper bounds on the width of the zeros of this analytic function. These can not be deduced in any straightforward way from the values of the potential $V$.
\end{remark}

 \section{Sometimes Two Limits Are Necessary}\label{sec:counterexample}
In Section \ref{sec:sci} we described a complicated construction called \emph{towers of algorithms}, where more than one successive limit is required to correctly perform certain computations.  In this section we exhibit this phenomenon first hand: we prove Theorem \ref{thm:main3} which is rephrased in the language of $\SCI$ as Theorem \ref{th:counterexample} below. This theorem shows that there exists a class of potentials (which are less smooth, yet can be evaluated at any point) for which there do \emph{not} exist algorithms that can approximate the associated spectral problem in a single limit. That is,  $\SCI>1$. We \emph{are} able, though,  to construct an arithmetic algorithm which converges by taking \emph{two} successive limits. That is, $\SCI=2$. This class contains potentials which are allowed to have a  singularity at a single point $x_0\in(0,1)$, but are otherwise smooth:
\begin{align*}
	\Om_{x_0}^{\mathrm{Sch}} &:= \left\{ V:\R\to\R \,\middle|\, V \text{ is 1-periodic, } V(x_0)=0 \text{ and }V|_{[0,1]}\in L^2\big([0,1]\big)\cap C^\infty\big([0,1]\setminus\{x_0\}\big) \right\}.
\end{align*}
We emphasize that $\Om_{x_0}^{\mathrm{Sch}}$ contains only real-valued functions and that the pointwise evaluation $V\mapsto V(x)$ is well-defined for all $x\in\R$. Moreover, we remark that by \cite[Th. XIII.96]{RS4} the Schr\"odinger operator $-\Delta+V$ is well-defined and selfadjoint with domain $H^2(\R)$. We shall therefore consider the computational problem

\begin{align}\label{eq:periodic_computational_problem-sch-x0}
\left\{
\begin{array}{rl}
	\hfill \Omega &=\quad  \Omega^{\mathrm{Sch}}_{x_0}\\[1mm]
	\mathcal M &=\quad \big(\{K\subset\C\,|\, K \text{ closed}\}, \, d_{\mathrm{AW}}\big)\\[1mm]
	\Lambda &=\quad \{V\mapsto V(x) \,|\, x\in\R^d \}\\[2mm]
	\Xi &: \quad\Om \to \mathcal M;\quad
	V\mapsto \sigma(-\partial_x^2+V),
\end{array}
\right.
\end{align}
and shall prove 
\begin{theorem}\label{th:counterexample}
The computational problem \eqref{eq:periodic_computational_problem-sch-x0} has $\SCI=2$ (equivalently, it belongs to $\Delta_3\setminus\Delta_2$).
\end{theorem}
The proof of Theorem \ref{th:counterexample} has two parts. To prove $\SCI\leq 2$ we construct an explicit tower of algorithms that computes the spectrum in two limits (cf. Definition \ref{def:Tower}). The proof of $\SCI> 1$ is by contradiction. We assume the existence of a sequence $\{\Gamma_n\}_{n\in\N}$ with $\Gamma_n(V)\to\sigma(-\partial_x^2+V)$ for every $V\in\Om_{x_0}^{\mathrm{Sch}}$ and via a diagonal process construct a potential $V\in\Om_{x_0}^{\mathrm{Sch}}$ such that $\Gamma_n(V)\nrightarrow\sigma(-\partial_x^2+V)$, yielding the desired contradiction.

\subsection{Lemmas}
We first collect  some technical lemmas that will be necessary for the proof.
\begin{lemma}\label{lemma:numerical_range}
	Denote by $W(H)$  the numerical range of an operator $H$. Let $V\in\Om_{x_0}^{\mathrm{Sch}}$. There exists $\delta>0$ such that $\inf W(-\partial_x^2+V) \geq -\f14$ whenever $\|V\|_{L^1([0,1])}<\delta$.
\end{lemma}
\begin{proof}
	Choose a partition of unity $\{\chi_i\}_{i\in\Z}$ such that $\supp(\chi_i)\subset [i-1,i+1]$ and $\sup_{i\in\Z}\|\chi_i\|_{W^{1,\infty}}<+\infty$. Then for any $\phi\in C_0^\infty(\R)$ we have
	\begin{align*}
		\int_\R V(x)|\phi(x)|^2 \,dx &= \sum_{i\in\Z} \int_{\R} V(x)\chi_i(x)|\phi(x)|^2 \,dx
		\\
		&= \sum_{i\in\Z} \int_{i-1}^{i+1} V(x)\chi_i(x)|\phi(x)|^2 \,dx
		\\
		&= \sum_{i\in\Z} \int_{-1}^{1} V(x)\chi_i(x-i)|\phi(x-i)|^2 \,dx,
	\end{align*}
	where the last line follows from periodicity of $V$. Moreover, we have
	\begin{align}
		\left| \int_{-1}^{1} V(x)\chi_i(x-i)|\phi(x-i)|^2 \,dx \right| 
		&\leq \|V\|_{L^1([-1,1])} \big\| \chi_i(\cdot-i)|\phi(\cdot-i)|^2 \big\|_{L^\infty([-1,1])}
		\nonumber
		\\
		&\leq C\|V\|_{L^1([0,1])} \big\|\chi_i(\cdot-i)|\phi(\cdot-i)|^2 \big\|_{W^{1,1}([-1,1])}
		\nonumber
		\\
		&\leq C \|V\|_{L^1([0,1])} \|\chi_i\|_{W^{1,\infty}} \|\phi(\cdot-i)\|_{H^1([-1,1])}^2,
		\label{eq:summing_inequality}
	\end{align}
	where the second line follows from the Sobolev embedding $W^{1,1}\hookrightarrow L^\infty$ (cf. \cite[Th. 8.8]{Brezis}) and the third follows by the product rule and H\"older's inequality. Summing \eqref{eq:summing_inequality} over $i$ we obtain
	\begin{align}
		\left|\int_\R V(x)|\phi(x)|^2 \,dx\right| 
		&\leq C \|V\|_{L^1([0,1])} \sum_{i\in\Z} \|\chi_i\|_{W^{1,\infty}} \|\phi(\cdot-i)\|_{H^1([-1,1])}^2
		\nonumber
		\\
		&\leq 2C \|V\|_{L^1([0,1])} \Big(\sup_{i\in\Z}\|\chi_i\|_{W^{1,\infty}}\Big) \|\phi\|_{H^1(\R)}^2
		\nonumber
		\\
		&\leq C' \|V\|_{L^1([0,1])} \|\phi\|_{H^1(\R)}^2,
		\label{eq:summed_inequality}
	\end{align}
	where $C' = 2C\sup_{i\in\Z}\|\chi_i\|_{W^{1,\infty}}$ is finite by choice of the functions $\chi_i$.
	Turning to the numerical range, we have for $\phi\in C_0^\infty(\R)$
	\begin{align*}
		\langle(-\partial_x^2+V)\phi,\phi\rangle_{L^2(\R)} &= \int_\R\left( |\phi'|^2 + V|\phi|^2\right)dx
		\\
		&\geq \|\phi'\|_{L^2(\R)}^2 - C'\|V\|_{L^1([0,1])} \|\phi\|_{H^1(\R)}^2
		\\
		&= (1-C'\|V\|_{L^1([0,1])})\|\phi'\|_{L^2(\R)}^2 - C'\|V\|_{L^1([0,1])}\|\phi\|_{L^2(\R)}^2,
	\end{align*}
	where the second line follows from \eqref{eq:summed_inequality}. Finally, let $\delta:=(4C')^{-1}$. Then if $\|V\|_{L^1([0,1])}<\delta$ one has
	\begin{align*}
		\langle(-\partial_x^2+V)\phi,\phi\rangle_{L^2(\R)} 
		\geq \f34 \|\phi'\|_{L^2(\R)}^2 - \f14 \|\phi\|_{L^2(\R)}^2
		\geq  - \f14 \|\phi\|_{L^2(\R)}^2
	\end{align*}
	and the proof is complete.
\end{proof}
\begin{lemma}\label{lemma:upper_bound_W}
	For $h>0$, $q_1,q_2\in[0,1]$ with $0<q_1<q_2<1$, define the periodic step function
	\begin{align*}
		V_{h,q_1,q_2}(x) = \begin{cases}
			-h & x\in[q_1,q_2]\\
			0 & x\in([0,1]\setminus[q_1,q_2]).
		\end{cases}
	\end{align*}
	Then for every $\tilde V\in\Om_{x_0}^{\mathrm{Sch}}$ with $\tilde V\leq 0$ one has
	\begin{align*}
		\inf W\big(-\partial_x^2+\tilde V + V_{h,q_1,q_2}\big) \leq -(q_2-q_1)h.
	\end{align*}
\end{lemma}
\begin{proof}
	Define the sequence of test functions
	\begin{align*}
		\phi_n(x)= \begin{cases}
			\f{1}{\sqrt{2n}} &x\in[-n,n]
			\\
			\text{linear to } 0 & x\in [-n-1,-n) \cup (n,n+1]\\
			0 & \textnormal{otherwise.}
		\end{cases}
	\end{align*}
	We note that $\|\phi_n\|_{L^2(\R)}\to 1$ as $n\to+\infty$.
	Then if $\tilde V\in\Om_{x_0}^{\mathrm{Sch}}$ with $\tilde V\leq 0$ we have
	\begin{align*}
		\langle(-\partial_x^2+V)\phi_n, \phi_n \rangle_{L^2(\R)}
		&= \int_{\R} |\phi_n'|^2\,dx + \int_\R (\tilde V + V_{h,q_1,q_2})|\phi_n|^2 \,dx
		\\
		&\leq \int_{\R} |\phi_n'|^2\,dx + \int_\R V_{h,q_1,q_2}|\phi_n|^2 \,dx
		\\
		&\leq \int_{-n-1}^{-n} \f{1}{2n} \,dx + \int_{n}^{n+1} \f{1}{2n}\,dx - \sum_{i=-n}^{n} \int_{[q_1,q_2]+i} \f{h}{2n}\,dx
		\\
		&= \f1n - (q_2-q_1)h.
	\end{align*}
	The assertion follows by letting $n\to+\infty$.
\end{proof}
The next lemma is needed to construct the tower of algorithms that will compute the spectrum of elements in $\Om^{\mathrm{Sch}}_{x_0}$.
\begin{lemma}\label{lemma:spectral_convergence}
	Let $V\in\Om_{x_0}^{\mathrm{Sch}}$ and for $n\in\N$ let $\rho_n\in W^{1,\infty}(\R)$ be a periodic function such that 
	\begin{align}\label{eq:regularized_potential}
		\rho_n(x) = \begin{cases}
			0 &  x\in \bigcup_{k\in\Z}(k-\tfrac1n, k+\tfrac1n) 
			\\
			1 &   x\in \R\setminus \bigcup_{k\in\Z}(k-\tfrac2n, k+\tfrac2n),
		\end{cases}
	\end{align}
	then one has $\sigma(-\partial_x^2+\rho_n V) \to \sigma(-\partial_x^2+V)$ in Attouch-Wets distance.
\end{lemma}
\begin{proof}

A Neumann series argument shows that
	\begin{align}\label{eq:neumann_series}
		(-\partial_x^2 + \rho_n V - z)^{-1} - (-\partial_x^2 + V - z)^{-1} = \sum_{m=1}^\infty \big( (\rho_n-1)V(-\partial_x^2 + V - z)^{-1} \big)^m,
	\end{align}
	where $z\notin\R$ and the right hand side is defined whenever $\|(\rho_n-1)V(-\partial_x^2 + V - z)^{-1}\|_{L^2\to L^2}<1$. We will prove that the right hand side of \eqref{eq:neumann_series} converges to 0 in the norm resolvent sense. To this end let $u\in L^2(\R)$ and  compute
	\begin{align}
		\left\| (\rho_n-1)V(-\partial_x^2 + V - z)^{-1} u \right\|_{L^2(\R)}^2 
		&= \sum_{k=-\infty}^\infty \left\| (\rho_n-1)V(-\partial_x^2 + V - z)^{-1} u \right\|_{L^2([k,k+1])}^2
		\nonumber 
		\\
		&\leq \sum_{k=-\infty}^\infty \left\| (\rho_n-1)V\right\|_{L^2([k,k+1])}^2 \left\|(-\partial_x^2 + V - z)^{-1} u \right\|_{L^\infty([k,k+1])}^2
		\nonumber 
		\\
		&= \left\| (\rho_n-1)V\right\|_{L^2([0,1])}^2 \sum_{k=-\infty}^\infty  \left\|(-\partial_x^2 + V - z)^{-1} u \right\|_{L^\infty([k,k+1])}^2
		\nonumber 
		\\
		&\leq C\left\| (\rho_n-1)V\right\|_{L^2([0,1])}^2 \sum_{k=-\infty}^\infty  \left\|(-\partial_x^2 + V - z)^{-1} u \right\|_{H^1([k,k+1])}^2
		\nonumber 
		\\
		&= C\left\| (\rho_n-1)V\right\|_{L^2([0,1])}^2   \left\|(-\partial_x^2 + V - z)^{-1} u \right\|_{H^1(\R)}^2
		\nonumber 
		\\
		&= C\left\| (\rho_n-1)V\right\|_{L^2([0,1])}^2   \left\|(-\partial_x^2 + V - z)^{-1}\right\|_{L^2(\R)\to H^1(\R)}^2 \|u\|_{L^2(\R)}^2,
		\label{eq:rhs_inequality}
	\end{align}
	where H\"older's inequality was used in the second line, periodicity of $(\rho_n-1)V$ was used in the third line and the Sobolev embedding $H^1(\R)\hookrightarrow L^\infty(\R)$ was used in the fouth line. Combining eqs. \eqref{eq:neumann_series} and \eqref{eq:rhs_inequality} we find that
	\begin{align*}
		\big\|(-\partial_x^2 + \rho_n V - z)^{-1} - & (-\partial_x^2 + V - z)^{-1}\big\|_{L^2(\R)\to L^2(\R)}
		\\
		&\leq C \sum_{m=1}^\infty \left\| (\rho_n-1)V\right\|_{L^2([0,1])}^m \left\|(-\partial_x^2 + V - z)^{-1}\right\|_{L^2(\R)\to H^1(\R)}^m ,
	\end{align*}
	for all $n>N$ large enough such that $\| (\rho_n-1)V\|_{L^2([0,1])} < \|(-\partial_x^2 + V - z)^{-1}\|_{L^2\to H^1}^{-1}$. Since $V\in L^2([0,1])$ we know that $\| (\rho_n-1)V\|_{L^2([0,1])}\to 0$, hence such $N$ must exist. For  $n>N$ the geometric series gives 
	\begin{align*}
		& \big\|(-\partial_x^2 + \rho_n V - z)^{-1} - (-\partial_x^2 + V - z)^{-1}\big\|_{L^2(\R)\to L^2(\R)} 
		\leq C \f{q_n}{1-q_n} ,
	\end{align*}
	with $q_n = \| (\rho_n-1)V\|_{L^2([0,1])}  \|(-\partial_x^2 + V - z)^{-1}\|_{L^2\to H^1}$, which immediately implies norm resolvent convergence. Finally, an application of \cite[Thms. VIII.23-24]{RS} yields the desired spectral convergence.
\end{proof}
\subsection{Proof of Theorem \ref{th:counterexample}}We are now ready to prove Theorem \ref{th:counterexample}. \\

\noindent\underline{{\bf Step 1}: construction of an adversarial potential $V$.}
	Assume for contradiction that there exists a sequence of algorithms $\{\Gamma_n\}_{n\in\N}$ such that  for any $V\in\Om_{x_0}^{\mathrm{Sch}}$, $\Gamma_n(V)\to\sigma(-\partial_x^2+V)$ as $n\to+\infty$. We now describe a process that defines an ``adversarial'' potential $V$ for which the sequence $\{\Gamma_n\}_{n\in\N}$ will necessarily fail. 
	
	We begin with an example construction that illustrates how a potential can be obtained that ``fools'' $\Gamma_n$ for a single $n$. This construction is then iterated below to obtain a potential whose spectrum is not approximated by the entire sequence $\{\Gamma_n\}_{n\in\N}$.
	For the sake of definiteness we make the arbitrary choice $x_0:=\f14$, though the proof does not depend on this choice. Moreover, we note that by H\"older's inequality one has $\|f\|_{L^1([0,1])} \leq \|f\|_{L^2([0,1])}$ for all $f\in L^2([0,1])$.
	
	With the notation from Lemma \ref{lemma:upper_bound_W}, let $V_1$ be the periodic square well potential $V_{h_1,0,\nicefrac12}$ with $h_1>0$ chosen such that $\inf W(-\partial_x^2+V_1)\leq -1$. Because $-\partial_x^2+V$ is selfadjoint, this implies that $\inf \sigma(-\partial_x^2+V_1)\leq -1$. Our assumptions about $\Gamma_n$ imply that there exists $n_1\in\N$ such that $\inf \re(\Gamma_n(V))<-\f12$ for all $n\geq n_1$. By the definition of an algorithm (Definition \ref{def:Algorithm}), $\Gamma_{n_1}(V)$ depends only on finitely many elements of $\Lambda$, i.e. finitely many point values $V(x_i)$, $i\in\{1,\dots,m_{n_1}\}$. We may assume without loss of generality that the sets $\{x_1,\dots,x_{m_n}\}$ are growing with $n$.
	
	For later reference, let $\rho$ be a smooth, compactly supported function on $\R$ such that
	\begin{align*}
		\rho(0) &= 1 
		\\
		\supp(\rho_1) &\subset\bigl(-\tfrac14, \tfrac14\bigr).
	\end{align*}
	Next, define a new potential $\tilde{V_1}$ by ``thinning out'' $V_1$ around the points $x_i$. More concretely, let $\delta>0$ (to be determined later) and define $l_1 := \min\big\{|x_i-x_j|\,\big|\, i,j\in \{0,\dots,m_{n_1}\},\,x_i\neq x_j\big\}$ (note that the point $x_0$ appears in the definition of $l_1$) and let
	\begin{align*}
		\tilde{V_1}(x) := V_1(x)\sum_{i=1}^{m_{n_1}}\rho\left(\f{x}{\delta l_1} - x_i\right).
	\end{align*}
	\begin{figure}[htbp]
		\centering
		\includegraphics{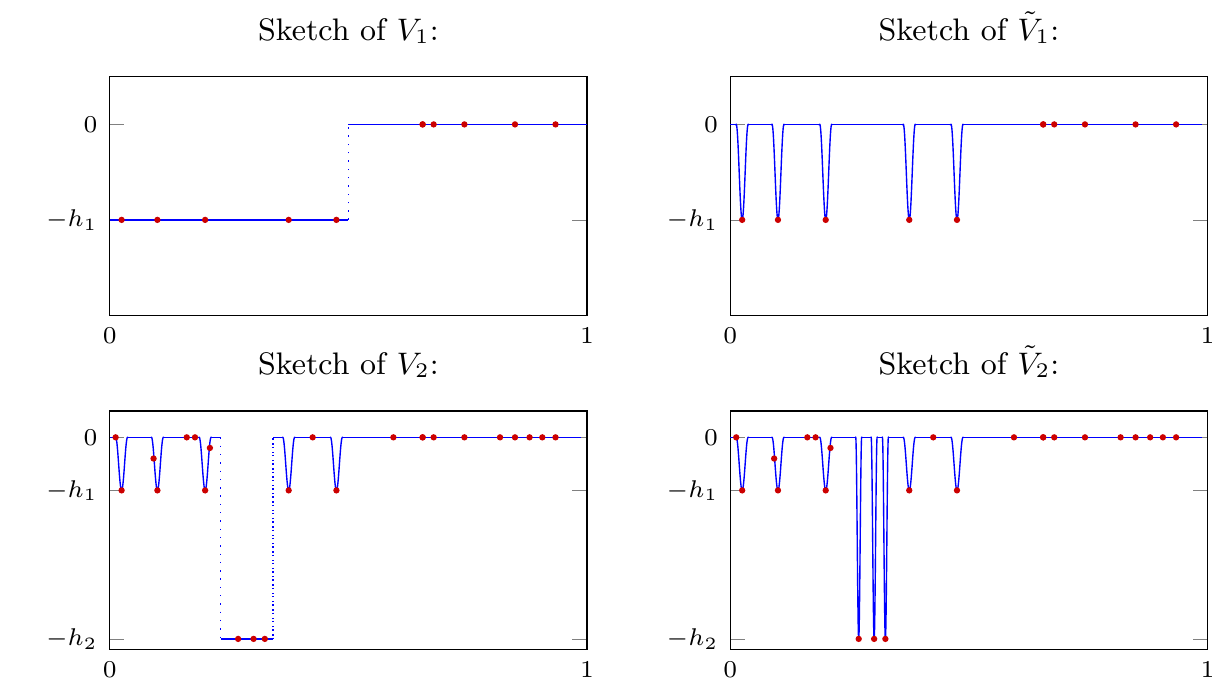}
		\caption{Sketch the first few iterations $V_k$, $\tilde V_k$. The red dots represent the points $x_i$, which represent the information in $\Lambda_{\Gamma_n}$.}
		\label{fig:conterexample}
	\end{figure}
	Then by construction we have $\|\tilde{V_1}\|_{L^2([0,1])} \leq \sum_{i=1}^{m_{n_1}} h_1 \sqrt{\delta l_1} \leq m_{n_1} h_1 \sqrt{\delta}$. Next, apply Lemma \ref{lemma:numerical_range} and choose $\delta$ such that $\inf W(-\partial_x^2+\tilde{V_1}) \geq -\f14$ (and hence $\inf \sigma(-\partial_x^2+\tilde{V_1}) \geq -\f14$). Note that we have $\tilde{V_1}(x_i)=V_1(x_i)$ for all $i\in\{0,\dots,m_{n_1}\}$. Hence by consistency of algorithms we have $\Gamma_{n_1}(\tilde{V_1}) = \Gamma_{n_1}(V_1)$. We have constructed a smooth potential such that 
	\begin{align*}
		\inf \re(\Gamma_n(\tilde{V_1})) &< -\f12
		\\
		\inf \sigma(-\partial_x^2+\tilde{V_1}) &> -\f14
	\end{align*}
	and thus $d_{\mathrm{AW}}\big(\Gamma_{n_1}(\tilde{V_1}), \sigma(-\partial_x^2+\tilde{V_1})\big) \geq \f14$. We remark that $\tilde V_1\equiv 0$ in a neighbourhood of $x_0$.
	
	The constructions outlined above define an iterative process that yields a sequence of smooth potentials $\{V_k\}_{k\in\N}$ and $\{\tilde V_k\}_{k\in\N}$. We outline the details below. Fix $\eta>0$ to be determined later and initialize $\tilde V_0 \equiv 0$.
	\begin{itemize}
		\item Let $2\leq k\in\N$ and suppose that $\tilde V_{k-1}$ has already been defined.
		\item Choose an interval $I=(x_0-\eps,x_0+\eps)$ on which $\tilde V_k\equiv 0$  and let $V_{k}:=\tilde V_{k-1} + V_{h_{k},x_0-\eps,x_0+\eps}$, where $h_{k}$ is chosen such that $\inf W(-\partial_x^2+V_{k})\leq -1$ (cf. Lemma \ref{lemma:upper_bound_W}).
		\item Choose $n_{k}$ large enough such that $\inf\re(\Gamma_{n_{k}}(V_{k}))<-\f12$. Then $\Gamma_{n_{k}}(V_{k})$ depends only on finitely many point values $x_1,\dots,x_{m_{n_{k}}}$. 
		\item Let $l_k := \min\big\{|x_i-x_j|\,\big|\, i,j\in \{0,\dots,m_{n_{k}}\},\,x_i\neq x_j\big\}$ and define a ``thinned out'' potential $\tilde V_{k}$ by 
		\begin{align*}
			\tilde{V}_{k}(x) := \tilde V_{k-1}(x) + V_{h_{k},x_0-\eps,x_0+\eps}(x)\sum_{i=1}^{m_{n_{k}}}\rho\left(\f{x}{\delta l_k} - x_i\right).
		\end{align*}
		Then $\|\tilde V_{k}\|_{L^2([0,1])}\leq \|\tilde V_{k-1}\|_{L^2([0,1])} + m_{n_{k}} h_{k} \sqrt{\delta}$.
		\item Choose $0<\delta<\left(\f{\eta}{2^{k}m_{n_{k}}h_{k}}\right)^2$ such that $\inf W(-\partial_x^2+\tilde{V}_{k})>-\f14$.
		\item It follows that $\Gamma_{n_{k}}(\tilde{V}_{k}) = \Gamma_{n_{k}}(V_{k})$ and thus $d_{\mathrm{AW}}\big(\Gamma_{n_{k}}(\tilde{V}_{k}), \sigma(-\partial_x^2+\tilde{V}_{k})\big) \geq \f14$.
		\item Moreover, we have by construction 
		\begin{align}\label{eq:geometric_sum}
			\|\tilde V_{k}\|_{L^2([0,1])} \leq \eta\sum_{j=0}^k 2^{-j}
		\end{align}
		and there exists an interval $(x_0-\eps',x_0+\eps')$ on which $\tilde V_{k}\equiv 0$.
	\end{itemize}
	This process defines a sequence of potentials $\{\tilde V_k\}_{k\in\N}$ such that for a subsequence $\{n_k\}_{k\in\N}$ one has 
	\begin{align}\label{eq:liminf_big}
		\liminf_{k\to+\infty} d_{\mathrm{AW}}\big(\Gamma_{n_k}(\tilde{V}_{k}), \sigma(-\partial_x^2+\tilde{V}_{k})\big) \geq \f14.
	\end{align}
	Next we show that the sequence $\tilde V_k$ converges pointwise to a function
	$V^\eta\in L^2([0,1])$, which is smooth on $[0,1]\setminus\{x_0\}$. By construction,  for every $\eps>0$, sequence $\{\tilde V_k|_{[0,1]\setminus (x_0-\eps,x_0+\eps)}\}_{k\in\N}$ is eventually constant. Combined with the fact that $\tilde V_k(x_0)=0$ for all $k$, this implies that $\{\tilde V_k\}_{k\in\N}$ converges pointwise to a function $V^\eta$ on $[0,1]$. Because for every $\eps>0$ there exists $k\in\N$ such that $V^\eta|_{[0,1]\setminus (x_0-\eps,x_0+\eps)} = \tilde V_k|_{[0,1]\setminus (x_0-\eps,x_0+\eps)}$, we have that $V^\eta$ is smooth on $[0,1]\setminus\{x_0\}$. Finally, by \eqref{eq:geometric_sum} for any $\eps>0$ there exists $k\in\N$ such that
	\begin{align*}
		\|V^\eta\|_{L^2([0,1]\setminus (x_0-\eps,x_0+\eps))} &= \|\tilde V_k\|_{L^2([0,1]\setminus (x_0-\eps,x_0+\eps))} 
		\\
		&\leq \eta\sum_{j=0}^k 2^{-j}
		\\
		&\leq 2\eta.
	\end{align*}
	Letting $\eps\to 0$ we conclude by monotone convergence that $V^\eta\in L^2([0,1])$ and consequently $V^\eta\in\Om_{x_0}^{\mathrm{Sch}}$. Moreover, the inequality $\|V^\eta\|_{L^2([0,1])}\leq 2\eta$ allows us to use Lemma \ref{lemma:numerical_range} and choose $\eta>0$ small enough that 
	\begin{align}\label{eq:inf_W_V^eta}
		\inf W(-\partial_x^2+V^\eta) = \inf \sigma(-\partial_x^2+V^\eta)\geq -\f14.
	\end{align}
	To conclude the proof we note that for any $k\in\N$ we have $V^\eta(x_i)=\tilde V_{k}(x_i)$ for all $i\in\{1,\dots,m_{n_k}\}$ and by consistency of algorithms we have
	\begin{align}\label{eq:V=Vk}
		\Gamma_{n_k}(V^\eta) = \Gamma_{n_k}(\tilde V_{k})
	\end{align}
	and thus $\inf\re(\Gamma_{n_k}(V^\eta)) \leq -\f12$. 
	Combining eqs. \eqref{eq:liminf_big}, \eqref{eq:inf_W_V^eta} and \eqref{eq:V=Vk} we conclude that
	\begin{align*}
		\liminf_{k\to+\infty} d_{\mathrm{AW}}\big(\Gamma_{n_k}(V^\eta), \sigma(-\partial_x^2+V^\eta)\big) \geq \f14.
	\end{align*}
	Consequently, the sequence $\Gamma_{n}(V^\eta)$ cannot converge to $\sigma(-\partial_x^2+V^\eta)$ and the desired contradiction follows, proving that $\SCI(\Om_{x_0}^{\mathrm{Sch}})\geq 2$.
	\\
	
	\noindent\underline{{\bf Step 2}: Construction of a tower of algorithms.} We conclude the proof of Theorem \ref{th:counterexample} by showing $\SCI(\Om_{x_0}^{\mathrm{Sch}})\leq 2$. To this end, choose a function $\rho_n$ as in \eqref{eq:regularized_potential} change this, if we change the $\rho_n$ above, whose values are explicitly computable (e.g. piecewise linear) and define the mapping
	\begin{align*}
		\Gamma_{m,n} &: \Om_{x_0}^{\mathrm{Sch}} \to \mathcal M
		\\
		\Gamma_{m,n}(V) &= \Gamma_m(\rho_n V),
	\end{align*}
	where $\Gamma_m$ denotes the algorithm from Definition \ref{def:final_alg}.
	Note that $\Gamma_m(\rho_n V)$ is well-defined because $\rho_n V\in W^{1,\infty}(\R)$ for every $n\in\N$. Applying Theorem \ref{th:main2} and Lemma \ref{lemma:spectral_convergence} we immediately find 
	\begin{align*}
		\lim_{n\to+\infty}\lim_{m\to+\infty} \Gamma_{m,n}(V) &= \lim_{n\to+\infty}\lim_{m\to+\infty} \Gamma_m(\rho_n V)
		\\
		&= \lim_{n\to+\infty} \sigma(-\partial_x^2+\rho_n V)
		\\
		&= \sigma(-\partial_x^2+V),
	\end{align*}
	where all limits are taken in Attouch-Wets distance. This completes the proof.\qed\\

	Note that the two limits in Step 2 above cannot be swapped, because it is unclear how $\Gamma_m(\rho_n V)$ behaves when $\rho_n V$ converges to a non-smooth function.

 
 \section{Numerical Results}\label{sec:numerics}

 To illustrate our abstract results, we implemented a version of algorithm \eqref{def:final_alg} in one dimension in Matlab. In this section we show the results of this implementation and compare them against known abstract and numerical results.

 In order to obtain an implementation with adequate performance, we fixed a box $Q$ in the complex plane and then computed the quantity
 \begin{align*}
 	\bigcup_{\theta\in \f1N\Z\cap[0,2\pi]}\left\{ z\in \f1N(\Z+\i\Z)\cap Q\,\middle|\, \big|\det\big(I-P_NK^{\mathrm{appr}}_{n}\big(z,\theta\big)P_N\big)\big|\leq C \right\},
 \end{align*}
 where the numbers $N$, $n$, $C$ were treated as independent parameters. Moreover, the spectral shift, which was chosen to be 1 in \eqref{eq:decomposition} and 2 in \eqref{eq:decomposition2} can be fixed to be any point $z_0$  outside the box $Q$. 
 The routine is illustrated by the following pseudocode.

\medskip
\begin{algorithm}[H]
	\setlength{\lineskip}{1mm}
	\SetAlgorithmName{Pseudocode}{}{}
	\SetAlgoInsideSkip{smallskip}
	\SetAlgoSkip{bigskip}
	\SetAlgoLined
	Fix $N,n\in\N$, $C>0$, $z_0\in\C$\;
	Define lattice $L_N\subset\C$\;
	Define lattice $\Theta_N\subset [0,2\pi]$\;
	Initialize spectrum $\sigma:=\{\}$\;
	Compute Fourier coefficients $\hat V_k^{\text{appr},n}$ (cf. \eqref{eq:approximated_fourier}) for $\f{k}{2\pi}\in\{-N,\dots, N\}$\;
	Define potential matrix $V^{\text{appr},n} := \big( \hat V^{\text{appr},n}_{k-k'} \big)_{k,k'\in\{-2\pi N,\dots,2\pi N\}}$\;
	Set $H:= \diag\big((1+|k|^2)^{\nf12},\,-N\leq \f{k}{2\pi}\leq N\big)$\;
	 \For{$\theta\in\Theta_N$}{
	 	 $B_0:=\diag( 2\theta k + |\theta|^2 - z_0 ,\; -N\leq \f{k}{2\pi}\leq N)$\;
	 	\For{$z\in L_n$}{
	 		 $R:=\diag\big( (z - z_0 - |k|^2)^{-1} ,\; -N\leq \f{k}{2\pi}\leq N\big)$\;
		 	 $K:=(B_0+ H^{-1}\cdot V^{\text{appr},n}\cdot H)\cdot R$\;
		 	 $D:=\det\big(I_{\scriptscriptstyle{(2N+1)\times (2N+1)}}+K\big)$\;
			\If{$|D|<C$}{
				 $\sigma := \sigma\cup \{z\}$\;
			}
		 }
	 }
	 \Return{$\sigma$}
	 \caption{Compute Spectrum}
	 \label{alg:compute_spectrum}
	\end{algorithm}

\medskip
The actual Matlab implementation of Pseudocode \ref{alg:compute_spectrum} is available online at \url{https://github.com/frank-roesler/PeriodicSpectra}.

\medskip
\noindent \textbf{Mathieu equation.}
 We consider the \emph{Mathieu equation}
 \begin{align}\label{eq:mathieu}
 	-u''(x) + \mu\cos(2\pi x)u(x) = \lambda u(x),
 \end{align}
 where $\mu\in\C$ is a constant and $\lambda$ denotes the spectral parameter. This equation was first studied in \cite{Mathieu} in the context of vibrating membranes and has been studied extensively since (see \cite[Ch. 5]{MF} for a discussion). Figure \ref{fig:mathieu} shows the output of our implementation for various values of $\mu$. 
 \begin{figure}[htbp]
 	\centering
 	\includegraphics{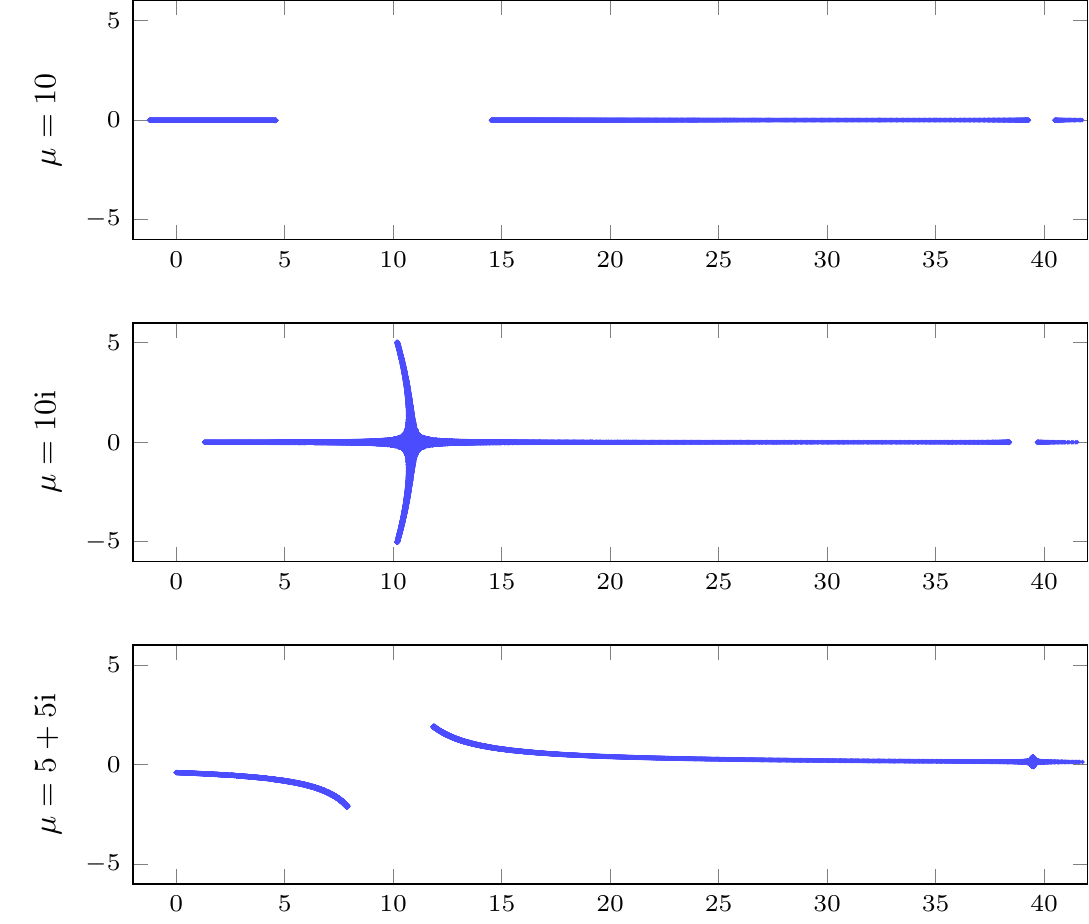}
 	\caption{Spectral approximation in $\C$ for the Mathieu operator for $\mu\in\{ 10,\;10\i,\,5+5\i\}$. Parameter values: $N=200$, $z_0=45$, $n=50$, $C=2\cdot 10^{-5}$.}
 	\label{fig:mathieu}
 \end{figure}
In the case of \textcolor{blue}{a} real-valued potential (top panel of Figure \ref{fig:mathieu}) our algorithm produces the expected band-gap structure, with one gap showing around $\lambda=10$ and another around $\lambda=40$. In the case of purely imaginary $\mu$, the theory of $\mathcal{PT}$-symmetric operators can be used to prove abstract results about the possible shape of the spectrum \cite{Shin}. A comparison between the middle panel of Figure \ref{fig:mathieu} and \cite[Fig. 2]{Shin} shows agreement between the theoretical results and the output of our algorithm. Finally, the bottom panel in Figure \ref{fig:mathieu} shows the output when $\mu$ has both a real and an imaginary part and the $\mathcal{PT}$ symmetry is broken.

To further validate our results, let us focus on the $\mathcal{PT}$ symmetric case ($\mu=10\i$) and compare them to existing results available in one dimension. Let $\phi_{1,\lambda}$, $\phi_{2,\lambda}$ be two classical solutions of \eqref{eq:mathieu} on $(0,1)$ with initial conditions $\phi_{1,\lambda}(0)=1$, $\phi_{1,\lambda}'(0)=0$ and $\phi_{2,\lambda}(0)=0$, $\phi_{2,\lambda}'(0)=1$. Then, the \emph{Hill Discriminant} is defined by
\begin{align*}
	D(\lambda) := \f12 (\phi_{1,\lambda}(1) + \phi_{2,\lambda}'(1))
\end{align*}
and one can show (cf. \cite{E73}) that $\lambda$ is in the spectrum of \eqref{eq:mathieu} if and only if $-1\leq D(\lambda)\leq 1$. Figure \ref{fig:discriminant} shows the points in $\C$ satisfying a softened version of this discriminant inequality, computed from a Runge-Kutta approximation of $\phi_{1,\lambda}$, $\phi_{2,\lambda}$. A comparison between Figures \ref{fig:discriminant} and \ref{fig:mathieu} shows good agreement.
\begin{figure}
	\centering
 	\includegraphics{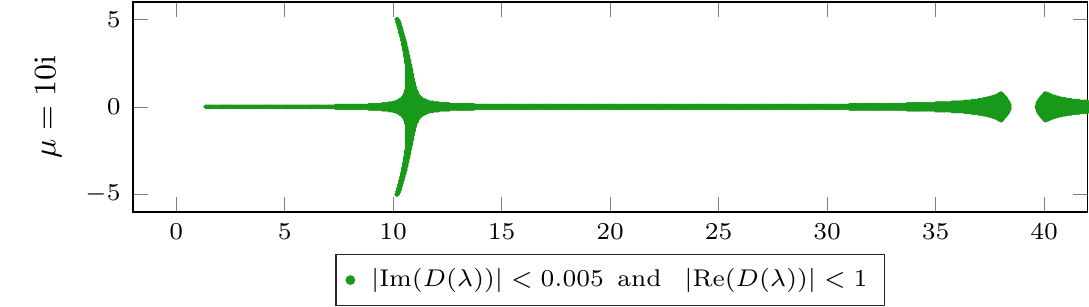}
 	\caption{Spectral approximation with discriminant method.}
 	\label{fig:discriminant}
\end{figure}

Finally, we note that our method is naturally immune to the common problem of \emph{spectral pollution}. By definition, spectral pollution occurs when there exist sequences $z_n\in\Gamma_n(V)$ such that $\{z_n\}_{n\in\N}$ has an accumulation point outside $\sigma(-\Delta+V)$. This effect appears in the approximation of Mathieu's equation if a naive finite difference scheme is used on a truncated domain (see Figure \ref{fig:domain_truncation}). For arbitrarily fine discretization and arbitrarily large truncated domain, the method is sensitive to small perturbations of the domain and can yield a set which is far from the correct spectrum.
\begin{figure}[htbp]
	\centering
 	\includegraphics{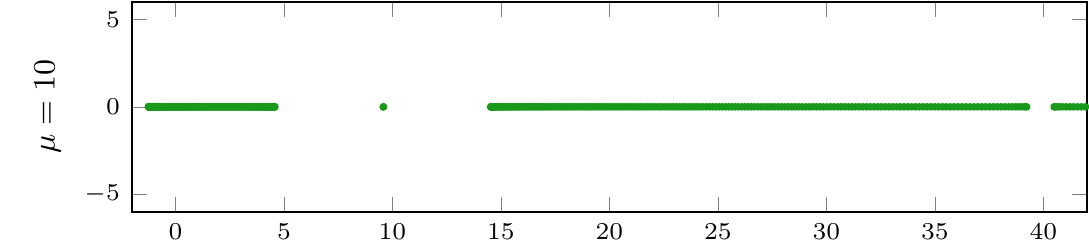}
 	\caption{Spectral approximation from finite difference scheme on truncated domain $[-100.7,\,100.7]$. The spurious eigenvalue in the gap for $\mu=10$ can appear for arbitrarily fine discretization.}
 	\label{fig:domain_truncation}
\end{figure}

 \appendix

\section{The Hausdorff and Attouch-Wets distances}
\label{sec:metrics}
Since the Attouch-Wets distance is not well-known, we provide the basic definition here, and use the opportunity to remind the reader of the Hausdorff distance as well. For further information we refer to \cite{beer}.

\begin{de}[Hausdorff distance for subsets of $\R^d$]
Let $A,B\subset\R^d$ be two non-empty bounded sets. Their \emph{Hausdorff distance} is
	\begin{align*}
	d_{\mathrm{H}}(A,B)
	&=
	\max\left\{\sup_{a\in A}\inf_{b\in B}|a-b|,\sup_{b\in B}\inf_{a\in A}|a-b|\right\}\\
	&=
	\sup_{p\in\R^d}\left|\inf_{a\in A}|a-p|-\inf_{b\in B}|b-p|\right|.
	\end{align*}
\end{de}

\begin{de}[Attouch-Wets distance for subsets of $\R^d$]\label{def:Attouch-Wets}
Let $A,B\subset\R^d$ be two non-empty (possibly unbounded) sets. Their  \emph{Attouch-Wets distance} is
	\begin{align*}
	d_{\mathrm{AW}}(A,B) = \sum_{n=1}^\infty 2^{-n}\min\left\{ 1\,,\,\sup_{p\in\R^d,\,|p|<n}\left| \inf_{a\in A}|a-p| - \inf_{b\in B}|b-p| \right| \right\}.
	\end{align*} 
\end{de}
Note that if $A,B$ are bounded, then $d_{\mathrm{AW}}$ and $d_\mathrm{H}$ are equivalent. Furthermore, it can be shown (cf. \cite[Ch. 3]{beer}) that 
	\begin{equation*}\label{eq:Attouch-Wets}
		d_\mathrm{H}(A_n\cap B, A\cap B)\to 0 \text{ for all }B\subset\R^d \text{ compact } \;\Rightarrow\; d_\mathrm{AW}(A_n,A)\to 0.
	\end{equation*}

\bibliography{mybib}
\bibliographystyle{abbrv}

\end{document}

%% file: preamble.tex
\newcommand{\f}{\frac}

\newcommand{\del}{\partial}

\newcommand{\im}{\operatorname{Im}}

\newcommand{\R}{\mathbb R}
\newcommand{\C}{\mathbb C}
\newcommand{\N}{\mathbb N}
\newcommand{\Z}{\mathbb Z}

\newcommand{\re}{\operatorname{Re}}
\newcommand{\eps}{\varepsilon}
\renewcommand{\epsilon}{\varepsilon}

\newcommand{\Om}{\Omega}

\newcommand{\supp}{\operatorname{supp}}
\newcommand{\dist}{\operatorname{dist}}

\newcommand{\dom}{\operatorname{dom}}
\newcommand{\h}{\mathcal{H}}

\newenvironment{enuma}{\begin{enumerate}[(a)]}{\end{enumerate}}
\newenvironment{enumi}{\begin{enumerate}[(i)]}{\end{enumerate}}

\AtBeginDocument{

}

\usepackage{etex}
\usepackage[english]{babel}
\usepackage{tikz}
\usetikzlibrary{arrows,decorations.pathmorphing,backgrounds,positioning,fit,petri,patterns}
\usepackage{pgfplots}
\usepackage[utf8]{inputenc}
\usepackage{graphicx}
\usepackage[colorlinks=true,linkcolor=red,citecolor=blue]{hyperref}
\usepackage{amsfonts}
\usepackage{enumerate}
\usepackage{booktabs}
\usepackage{array} 
\usepackage{paralist} 
\usepackage{subfig} 
\usepackage{mathtools}
\usepackage{tabu}
\usepackage{amsthm}
\usepackage{empheq}
\usepackage{amsopn}
\usepackage{dsfont}
\usepackage{wrapfig}
\usepackage[font=small]{caption}
\usepackage{units}
\usepackage[symbol]{footmisc}
\usepackage{todonotes}
\usepackage{amssymb}
\usepackage{pdfpages}
\usepackage{setspace}
\usepackage{float}

\makeatletter
\renewenvironment{cases}[1][l]{\matrix@check\cases\env@cases{#1}}{\endarray\right.}
\def\env@cases#1{%
  \let\@ifnextchar\new@ifnextchar
  \left\lbrace\def\arraystretch{1.2}%
  \array{@{}#1@{\quad}l@{}}}
\makeatother

\usepackage{kvsetkeys}
\usepackage{etexcmds}

\makeatletter
\newcount\arg@count
\newcommand{\arg@parser}[1]{%
  \advance\arg@count\@ne
  \expandafter\let\csname arg\romannumeral\arg@count\endcsname\comma@entry
}
\newcommand\res[1]{
  \arg@count=\z@
  \comma@parse{ \lambda,A }\arg@parser 
  \arg@count=\z@
  \comma@parse{#1}\arg@parser
  \ifnum\arg@count>2 %
    \@latex@error{Too many arguments}{%
      The macro \string\mycmd\space got \the\arg@count\space
       arguments,\MessageBreak
      but expected are 2 arguments.\MessageBreak
      \@ehd
    }%
  \fi
  \edef\process@me{%
    \noexpand\@res
    {\etex@unexpanded\expandafter{\argi}}%
    {\etex@unexpanded\expandafter{\argii}}%
  }%
  \process@me
}
\newcommand{\@res}[2]{%
  \ensuremath\left( #1 - #2 \right)^{-1}
}
\makeatother

\makeatletter
\newcount\arg@count
\newcommand\p[1]{
  \arg@count=\z@
  \comma@parse{  }\arg@parser 
  \arg@count=\z@
  \comma@parse{#1}\arg@parser
  \ifnum\arg@count=2 %
  \else
    \@latex@error{Wrong number of mandatory arguments}{%
      The macro \string\p\space got \the\numexpr\arg@count-2\relax\space
      mandatory arguments,\MessageBreak
      but expected are 3 mandatory arguments.\MessageBreak
      \@ehd
    }%
  \fi
  \edef\process@me{%
    \noexpand\@p
    {\etex@unexpanded\expandafter{\argi}}%
    {\etex@unexpanded\expandafter{\argii}}%
  }%
  \process@me
}
\newcommand{\@p}[2]{%
  \ensuremath \left\langle #1 , #2 \right\rangle
}
\makeatother

\allowdisplaybreaks
\numberwithin{equation}{section}

\theoremstyle{definition}
\newtheorem{de}{Definition}[section]

\theoremstyle{plain}
\newtheorem{prop}[de]{Proposition}
\newtheorem{lemma}[de]{Lemma}
\newtheorem{theorem}[de]{Theorem}

\newtheorem{corollary}[de]{Corollary}

\theoremstyle{remark}
\newtheorem{remark}[de]{Remark}
\newtheorem{example}[de]{Example}

%% file: periodic_main.bbl
\begin{thebibliography}{10}

\bibitem{avila2009ten}
A.~Avila and S.~Jitomirskaya.
\newblock The ten martini problem.
\newblock {\em Ann. Math.}, 170(1):303--342, 2009.

\bibitem{Becker2020b}
S.~Becker and A.~C. Hansen.
\newblock {Computing solutions of Schr\"odinger equations on unbounded domains-
  On the brink of numerical algorithms}.
\newblock {\em arXiv e-prints, 2010.16347}, 2020.

\bibitem{beer}
G.~Beer.
\newblock {\em Topologies on closed and closed convex sets}, volume 268 of {\em
  Mathematics and its Applications}.
\newblock Kluwer Academic Publishers Group, Dordrecht, 1993.

\bibitem{AHS}
J.~Ben{-}Artzi, M.~J. Colbrook, A.~C. Hansen, O.~Nevanlinna, and M.~Seidel.
\newblock Computing spectra -- {O}n the solvability complexity index hierarchy
  and towers of algorithms.
\newblock {\em arXiv e-prints, 1508.03280}, {A}ug 2015.

\bibitem{Ben-Artzi2015a}
J.~Ben-Artzi, A.~C. Hansen, O.~Nevanlinna, and M.~Seidel.
\newblock {New barriers in complexity theory: On the solvability complexity
  index and the towers of algorithms}.
\newblock {\em Comptes Rendus Math.}, 353(10):931--936, {O}ct 2015.

\bibitem{BMR2020}
J.~{Ben-Artzi}, M.~{Marletta}, and F.~{R\"{o}sler}.
\newblock {Computing Scattering Resonances}.
\newblock {\em arXiv e-prints, 2006.03368}, June 2020.

\bibitem{Ben-Artzi2020a}
J.~Ben-Artzi, M.~Marletta, and F.~R{\"{o}}sler.
\newblock {Computing the Sound of the Sea in a Seashell}.
\newblock {\em Foundations of Computational Mathematics (accepted)}, sep 2020.

\bibitem{Brezis}
H.~Brezis.
\newblock {\em Functional analysis, {S}obolev spaces and partial differential
  equations}.
\newblock Universitext. Springer, New York, 2011.

\bibitem{MR0069338}
E.~A. Coddington and N.~Levinson.
\newblock {\em Theory of ordinary differential equations}.
\newblock McGraw-Hill Book Company, Inc., New York-Toronto-London, 1955.

\bibitem{Colbrook2019a}
M.~J. Colbrook.
\newblock {On the computation of geometric features of spectra of linear
  operators on Hilbert spaces}.
\newblock {\em arXiv e-prints, 1908.09598}, 2019.

\bibitem{Colbrook2019}
M.~J. Colbrook and A.~C. Hansen.
\newblock {The foundations of spectral computations via the Solvability
  Complexity Index hierarchy: Part I}.
\newblock {\em arXiv e-print, 1908.09592}, 2019.

\bibitem{Colbrook2019c}
M.~J. Colbrook, B.~Roman, and A.~C. Hansen.
\newblock {How to Compute Spectra with Error Control}.
\newblock {\em Physical Review Letters}, 122(25):250201, 2019.

\bibitem{DS2}
N.~Dunford and J.~T. Schwartz.
\newblock {\em Linear operators. {P}art {II}}.
\newblock Wiley Classics Library. John Wiley \& Sons, Inc., New York, 1988.
\newblock Spectral theory. Selfadjoint operators in Hilbert space, With the
  assistance of William G. Bade and Robert G. Bartle, Reprint of the 1963
  original, A Wiley-Interscience Publication.

\bibitem{E73}
M.~S.~P. Eastham.
\newblock {\em The spectral theory of periodic differential equations}.
\newblock Texts in Mathematics (Edinburgh). Scottish Academic Press, Edinburgh;
  Hafner Press, New York, 1973.

\bibitem{MR1257824}
A.~Figotin and P.~Kuchment.
\newblock Band-gap structure of the spectrum of periodic {M}axwell operators.
\newblock {\em J. Statist. Phys.}, 74(1-2):447--455, 1994.

\bibitem{MR1372891}
A.~Figotin and P.~Kuchment.
\newblock Band-gap structure of spectra of periodic dielectric and acoustic
  media. {I}. {S}calar model.
\newblock {\em SIAM J. Appl. Math.}, 56(1):68--88, 1996.

\bibitem{MR1417473}
A.~Figotin and P.~Kuchment.
\newblock Band-gap structure of spectra of periodic dielectric and acoustic
  media. {II}. {T}wo-dimensional photonic crystals.
\newblock {\em SIAM J. Appl. Math.}, 56(6):1561--1620, 1996.

\bibitem{MR2909914}
S.~Giani and I.~G. Graham.
\newblock Adaptive finite element methods for computing band gaps in photonic
  crystals.
\newblock {\em Numer. Math.}, 121(1):31--64, 2012.

\bibitem{Gohberg2003}
I.~Gohberg, S.~Goldberg, and M.~A. Kaashoek.
\newblock {Laurent and Toeplitz Operators}.
\newblock In {\em Basic Classes of Linear Operators}, pages 135--170.
  Birkh{\"{a}}user Basel, Basel, 2003.

\bibitem{Hansen11}
A.~C. Hansen.
\newblock On the solvability complexity index, the {$n$}-pseudospectrum and
  approximations of spectra of operators.
\newblock {\em J. Amer. Math. Soc.}, 24(1):81--124, 2011.

\bibitem{MR1609321}
B.~Helffer and A.~Mohamed.
\newblock Asymptotic of the density of states for the {S}chr\"{o}dinger
  operator with periodic electric potential.
\newblock {\em Duke Math. J.}, 92(1):1--60, 1998.

\bibitem{MR2556599}
V.~Hoang, M.~Plum, and C.~Wieners.
\newblock A computer-assisted proof for photonic band gaps.
\newblock {\em Z. Angew. Math. Phys.}, 60(6):1035--1052, 2009.

\bibitem{last2005spectral}
Y.~Last.
\newblock Spectral theory of sturm-liouville operators on infinite intervals: a
  review of recent developments.
\newblock {\em Sturm-Liouville Theory}, pages 99--120, 2005.

\bibitem{Mathieu}
{\'E}.~Mathieu.
\newblock M\'{e}moire sur le mouvement vibratoire d'une membrane de forme
  elliptique.
\newblock {\em Journal de math{\'e}matiques pures et appliqu{\'e}es},
  13:137--203, 1868.

\bibitem{MF}
P.~M. Morse and H.~Feshbach.
\newblock Methods of theoretical physics.
\newblock {\em American Journal of Physics}, 22(6):410--413, 1954.

\bibitem{MR2419769}
L.~Parnovski.
\newblock Bethe-{S}ommerfeld conjecture.
\newblock {\em Ann. Henri Poincar\'{e}}, 9(3):457--508, 2008.

\bibitem{MR629118}
V.~N. Popov and M.~M. Skriganov.
\newblock Remark on the structure of the spectrum of a two-dimensional
  {S}chr\"{o}dinger operator with periodic potential.
\newblock {\em Zap. Nauchn. Sem. Leningrad. Otdel. Mat. Inst. Steklov. (LOMI)},
  109:131--133, 181, 183--184, 1981.
\newblock Differential geometry, Lie groups and mechanics, IV.

\bibitem{RS4}
M.~Reed and B.~Simon.
\newblock {\em Methods of modern mathematical physics. {IV}. {A}nalysis of
  operators}.
\newblock Academic Press, New
  York-London, 1978.

\bibitem{RS}
M.~Reed and B.~Simon.
\newblock {\em Methods of modern mathematical physics. {I}}.
\newblock Academic Press, New
  York, second edition, 1980.
\newblock Functional analysis.

\bibitem{MR0157033}
F.~S. Rofe-Beketov.
\newblock On the spectrum of non-selfadjoint differential operators with
  periodic coefficients.
\newblock {\em Dokl. Akad. Nauk SSSR}, 152:1312--1315, 1963.

\bibitem{R19}
F.~R\"{o}sler.
\newblock On the solvability complexity index for unbounded selfadjoint and
  schr{\"o}dinger operators.
\newblock {\em Integral Equations and Operator Theory}, 91(6):54, 2019.

\bibitem{Shin}
K.~C. Shin.
\newblock On the shape of spectra for non-self-adjoint periodic
  {S}chr\"{o}dinger operators.
\newblock {\em J. Phys. A}, 37(34):8287--8291, 2004.

\bibitem{Simon}
B.~Simon.
\newblock Notes on infinite determinants of {H}ilbert space operators.
\newblock {\em Advances in Math.}, 24(3):244--273, 1977.

\bibitem{MR784531}
M.~M. Skriganov.
\newblock The spectrum band structure of the three-dimensional
  {S}chr\"{o}dinger operator with periodic potential.
\newblock {\em Invent. Math.}, 80(1):107--121, 1985.

\bibitem{BS-Conj}
A.~Sommerfeld and H.~Bethe.
\newblock {Elektronentheorie der Metalle}.
\newblock In {\em Aufbau Der Zusammenh{\"{a}}ngenden Materie}, volume~19 of
  {\em Heidelberger Taschenb{\"{u}}cher}, pages 333--622. Springer Berlin
  Heidelberg, Berlin, Heidelberg, 1933.

\bibitem{teschl2000jacobi}
G.~Teschl.
\newblock {\em Jacobi Operators and Completely Integrable Nonlinear Lattices}.
\newblock Mathematical surveys and monographs. American Mathematical Society,
  2000.

\bibitem{TE}
L.~N. Trefethen and M.~Embree.
\newblock {\em {S}pectra and {P}seudospectra}.
\newblock Princeton University Press, Princeton, NJ, 2005.
\newblock The behavior of nonnormal matrices and operators.

\end{thebibliography}
